\documentclass[12pt]{amsart}
\usepackage{}

\usepackage{amsmath}
\usepackage{amsfonts}
\usepackage{amssymb}
\usepackage[all]{xy}           %xypic macro for latex2.09

\usepackage{bbding}
\usepackage{txfonts}
\usepackage{amscd}

\usepackage[shortlabels]{enumitem}
\usepackage{ifpdf}
\ifpdf
  \usepackage[colorlinks,final,backref=page,hyperindex]{hyperref}
\else
  \usepackage[colorlinks,final,backref=page,hyperindex,hypertex]{hyperref}
\fi
\usepackage{tikz}
\usepackage[active]{srcltx}

\makeatletter

\newtheorem{thm}{Theorem}[section]
\newtheorem{lem}[thm]{Lemma}
\newtheorem{cor}[thm]{Corollary}
\newtheorem{pro}[thm]{Proposition}
\newtheorem{ex}[thm]{Example}
\newtheorem{rmk}[thm]{Remark}
\newtheorem{defi}[thm]{Definition}

\newcommand {\emptycomment}[1]{}

\newcommand{\lon }{\,\rightarrow\,}
\newcommand{\be }{\begin{equation}}
\newcommand{\ee }{\end{equation}}

\newcommand{\g}{\mathfrak g}

\newcommand{\huaB}{\mathcal{B}}%{{\mathcal{E}}}%{\mathcal{B}}

\newcommand{\huaH}{\mathcal{H}}

\newcommand{\huaO}{{\mathcal{O}}}

\newcommand{\huaZ}{\mathcal{Z}}

\newcommand{\frkT}{\mathfrak T}

\newcommand{\Courant}[1]{\left\llbracket  #1\right\rrbracket }

\newcommand{\Id}{{\rm{Id}}}

\newcommand{\br}[1]{   [ \cdot,    \cdot  ]   }

\newcommand{\Hom}{\mathrm{Hom}}

\newcommand{\gl}{\mathfrak {gl}}

\newcommand{\ad}{\mathrm{ad}}

\newcommand{\U}{\mathrm{U}}

\newcommand{\Li}{\mathsf{3Lie}}

\begin{document}

\title{Twisted Rota-Baxter operators on $3$-Lie algebras and NS-$3$-Lie algebras}

\author{Shuai Hou}
\address{Department of Mathematics, Jilin University, Changchun 130012, Jilin, China}
\email{hshuaisun@163.com}

\author{Yunhe Sheng}
\address{Department of Mathematics, Jilin University, Changchun 130012, Jilin, China}
\email{shengyh@jlu.edu.cn}

%\date{\today}

\begin{abstract}
In this paper, first we introduce the notion of a twisted Rota-Baxter operator  on a $3$-Lie algebra $\g$ with   a representation on $V$. We show that a twisted Rota-Baxter operator induces a 3-Lie algebra structure on $V$, which represents on $\g$. By this fact, we define the cohomology of a twisted Rota-Baxter operator  and study infinitesimal deformations of a twisted Rota-Baxter operator using the second cohomology group. Then we introduce  the notion of an NS-$3$-Lie algebra, which produces a 3-Lie algebra with a representation on itself. We show that   a twisted Rota-Baxter operator induces an NS-$3$-Lie algebra naturally. Thus NS-$3$-Lie algebras can be viewed as the underlying algebraic structures of twisted Rota-Baxter operators  on   $3$-Lie algebras. Finally we show that a Nijenhuis operator on a 3-Lie algebra gives rise to a representation of the deformed 3-Lie algebra and a 2-cocycle. Consequently, the identity map will be a twisted Rota-Baxter operator on the deformed 3-Lie algebra. We also introduce the notion of a  Reynolds operator  on a $3$-Lie algebra, which  can serve as  a special case  of twisted Rota-Baxter operators on 3-Lie algebras.
\end{abstract}

%\subjclass[2010]{17A42, 17B38  17B40, 17B56}

\keywords{Twisted-Rota Baxter operator, NS-$3$-Lie algebra, Nijenhuis operator, Reynolds operator\\ \emph{\qquad 2020 Mathematics Subject Classification.}17A42, 17B38,  17B40, 17B56\\
Corresponding author: Yunhe Sheng}

\maketitle

\tableofcontents

\allowdisplaybreaks

%\end{document}

\section{Introduction}
\subsection{Cohomology of  Rota-Baxter operators on Lie algebras and 3-Lie algebras}
The notion of a Rota-Baxter operator on an associative algebra was introduced in the probability study of G. Baxter and later found important applications in the Connes-Kreimer's algebraic approach to renormalization of quantum field theory \cite{CK}. For further details, see ~\cite{Gub-AMS,Gub}. A Rota-Baxter operator on a Lie algebra is naturally the operator form of a classical $r$-matrix~\cite{STS} under certain conditions.
To better understand such connection in general, Kupershmidt introduced the notion of an $\huaO$-operator (also called
a relative Rota-Baxter operator \cite{PBG} or a generalized Rota-Baxter operator \cite{Uch})
on a Lie algebra in~\cite{Ku}. Recently the   deformation theory and the cohomology theory of relative Rota-Baxter operators on both Lie and associative algebras are studied  in \cite{Das,TBGS}.

 3-Lie
algebras and more generally, $n$-Lie
algebras (also called Filippov algebras)~\cite{Filippov}, have attracted attention
from both mathematics and physics. % The $n$-Lie algebra is the algebraic structure corresponding to Nambu mechanics \cite{N}.
See the review article \cite{review,Makhlouf} for more details.
In \cite{BaiRGuo}, The authors introduced the notion of a Rota-Baxter operator  on a 3-Lie algebra and studied its relation with Rota-Baxter Lie algebras.  To construct solutions of the classical 3-Lie Yang-Baxter equation, the authors introduced a more general notion,
an $\huaO$-operator (also called a relative Rota-Baxter operator) on a 3-Lie algebra with respect to a representation  in \cite{BGS-3-Bialgebras}.  In \cite{THS}, the authors constructed a Lie 3-algebra whose Maurer-Cartan elements are relative Rota-Baxter operators on 3-Lie algebras, and studied cohomologies and deformations of relative Rota-Baxter operators on 3-Lie algebras. See \cite{deformation} for more details about deformations of 3-Lie algebras.

\subsection{Nijenhuis operators,  Reynolds operators and NS-algebras}
 Nijenhuis operators  play an important role in  deformation
theories  due to their relationship with trivial infinitesimal
deformations. There are interesting applications of Nijenhuis operators on associative algebras and Lie algebras such as constructing biHamiltonian systems to study the integrability of nonlinear evolution equations \cite{CGM,Do}. The notion of a Nijenhuis operator on an $n$-Lie algebra was introduced in \cite{Liu-Jie-Feng} to study deformations of $n$-Lie algebras. Then it was applied to study product structures and complex structures on a 3-Lie algebra in \cite{Sheng-Tang}. See \cite{Zhangtao} for a different notion of Nijenhuis operators on 3-Lie algebras.

 Reynolds operators were introduced by Reynolds in \cite{Re} in the study of fluctuation theory in fluid dynamics.
In \cite{KAM}, the author coined the concept of the Reynolds operator and regarded the operator as a mathematical subject in general. %Reynolds operators play important roles in the study---
In \cite{gao-guo}, the authors provided examples and properties of the Reynolds operators and studied the free Reynolds algebras.

 NS-algebras were introduced by Leroux in \cite{Leroux}. In \cite{LG}, the authors studied   the relationship between the category of Nijenhuis algebras and the category of NS-algebras. Uchino defined the notion of twisted Rota-Baxter operators on associative algebras and observed that a twisted Rota-Baxter operator induces an NS-algebra in \cite{Uch}.  In \cite{Das-2}, Das studied the cohomology of twisted Rota-Baxter operators on associative algebras and the cohomology of NS-algebras and gave various application.   Das also introduced the notions of twisted Rota-Baxter operators, Reynolds operators on Lie algebras and NS-Lie algebras, and showed that NS-Lie algebras are the underlying algebraic structures of twisted Rota-Baxter operators on Lie algebras  in \cite{Das-1}.

\subsection{Main results and outline of the paper}
 In this paper, we introduce the notion of a twisted Rota-Baxter operator (by a 2-cocycle)  on a 3-Lie algebra, and study its cohomology theory and the underlying algebraic structure. Note that a twisted Rota-Baxter operator $T:V\lon\g$ on a 3-Lie algebra $\g$ with respect to a representation on $V$ naturally induces a 3-Lie algebra structure on $V$ with a representation on $\g$. The corresponding cohomology is defined to be the cohomology of the relative Rota-Baxter operator. As applications, we use the second cohomology group to study infinitesimal deformations of twisted Rota-Baxter operators on 3-Lie algebras. We introduce the notion of an NS-3-Lie algebra and show that a twisted Rota-Baxter operator on a 3-Lie algebra naturally induces an NS-3-Lie algebra. Thus, NS-3-Lie algebras can be viewed as the underlying algebraic structures of twisted Rota-Baxter operators on   3-Lie algebras. We study two special classes of twisted Rota-Baxter operators on  3-Lie algebras:
 \begin{itemize}
   \item We show   that the identity map is a twisted Rota-Baxter operator  on the deformed  3-Lie algebra by a Nijenhuis operator;
   \item We introduce the notion of a    Reynolds operator  on a $3$-Lie algebra, which  turns out to be  a special   twisted Rota-Baxter operator.
 \end{itemize}
 We also give   concrete examples of Reynolds operators on some infinite dimensional 3-Lie algebras including the $\omega_{\infty}$~$3$-Lie algebra.

 The paper is organized as follows. In Section \ref{sec:L}, we introduce the notion of a twisted Rota-Baxter operator on a 3-Lie algebra and establish the corresponding cohomology theory. Applications are given to study infinitesimal deformations.  In Section \ref{sec:GM}, we introduce the notion of an NS-3-Lie algebra and show that a twisted Rota-Baxter operator on a 3-Lie algebra naturally induces an NS-3-Lie algebra. In Section \ref{sec:N}, first we show that a Nijenhuis operator on a 3-Lie algebra induces a representation of the deformed 3-Lie algebra and a 2-cocycle. Consequently, the identity map is a twisted Rota-Baxter operator on the deformed 3-Lie algebra. Then we introduce the notion of a   Reynolds operator  on a $3$-Lie algebra, which  turns out to be  a special   twisted Rota-Baxter operator. Finally we give   examples of Reynolds operators  on some  infinite dimensional 3-Lie algebras.

%\vspace{2mm}

%{\bf Convention.} In this paper, we work over an algebraically closed field $\K$ of characteristic 0 and all the vector spaces are over $\K$ and finite-dimensional.

\vspace{2mm}
\noindent
{\bf Acknowledgements. } This research is supported by NSFC
(11922110).

%\vspace{2mm}
%\noindent
%{\bf Data Availability.} The data that support the findings of this study are available from the corresponding author upon reasonable request.

\section{Cohomologies  of twisted Rota-Baxter operators on $3$-Lie algebras}\label{sec:L}

\subsection{Twisted Rota-Baxter operators on $3$-Lie algebras}
In this subsection, we introduce the notion of a $\Phi$-twisted Rota-Baxter operator on a 3-Lie algebra and show that a linear map $T:V\rightarrow\g$ is a $\Phi$-twisted Rota-Baxter operator if and only if the graph of $T$ is a subalgebra of the $\Phi$-twisted semidirect product 3-Lie algebra $\g\ltimes_{\Phi} V$, where $\Phi$ is a 2-cocycle of the 3-Lie algebra $\g$ with coefficients in $V$. Consequently, a $\Phi$-twisted Rota-Baxter operator $T$ induces a $3$-Lie algebra structure on $V.$ We also use a $T$-admissible $1$-cocycle $f$ to construct a new $\Phi$-twisted Rota-Baxter operator $T_f$, and the $3$-Lie algebra structures on $V$ induced by $T$ and $T_{f}$ are isomorphic.

\begin{defi}{\rm (\cite{Filippov})}
A {\bf 3-Lie algebra} is a vector space $\g$ together with a skew-symmetric linear map $[\cdot,\cdot,\cdot]_{\g}:\otimes^{3}\g\rightarrow \g$, such that for $ x_{i}\in \g, 1\leq i\leq 5$, the following {\bf Fundamental Identity} holds:
\begin{equation}\label{eq:jacobi1}
~[x_1,x_2,[x_3,x_4, x_5]_{\g}]_{\g}=[[x_1,x_2, x_3]_{\g},x_4,x_5]_{\g}+[x_3,[x_1,x_2, x_4]_{\g},x_5]_{\g}+[x_3,x_4,[x_1,x_2, x_5]_{\g}]_{\g}.
\end{equation}
\end{defi}

A {\bf derivation} on a $3$-Lie algebra $(\g,[\cdot,\cdot,\cdot]_{\g})$ is a linear map $D:\g\rightarrow \g$ satisfying
\begin{eqnarray}
 D[x_1,x_2,x_2]_{\g}=[Dx_1,x_2,x_3]_{\g}+[x_1,Dx_2,x_3]_{\g}+[x_1,x_2,Dx_3]_{\g}, \quad \forall x_1,x_2,x_3\in \g.
\label{eq:der}
\end{eqnarray}

For $x_1,x_2\in \g,$ define $\ad:\wedge^2\g\rightarrow \gl(\g)$ by
$$\ad_{x_1,x_2}x_3=[x_1,x_2,x_3]_{\g},\quad \forall x_3\in \g.$$
Then \eqref{eq:jacobi1} is equivalent to that $\ad_{x_1,x_2}$ is a derivation, i.e.
$$\ad_{x_1,x_2}[x_3,x_4,x_5]_{\g}=[\ad_{x_1,x_2}x_3,x_4,x_5]_{\g}+[x_3,\ad_{x_1,x_2}x_4,x_5]_{\g}+[x_3,x_4,\ad_{x_1,x_2}x_5]_{\g}.$$

\begin{defi}{\rm (\cite{KA})}\label{defi-representation}
 A {\bf representation} of a $3$-Lie algebra $(\g,[\cdot,\cdot,\cdot]_{\g})$ is a pair $(V; \rho)$, where $V$ is a vector space, and $\rho:\wedge^{2}\g \rightarrow \gl(V)$ is a linear map satisfying, for all $x_{1}, x_{2}, x_{3}, x_{4}\in \g,$
\begin{eqnarray}
 \label{representation-1}[\rho(x_{1},x_{2}),\rho(x_{3},x_{4})] &=&\rho([x_{1},x_{2},x_{3}]_{\g},x_{4})+\rho(x_{3},[x_{1},x_{2},x_{4}]_{\g}),\\
  \label{representation-2}\rho([x_{1},x_{2},x_{3}]_{\g},x_{4})&=&\rho (x_{1},x_{2})\rho (x_{3},x_{4})+\rho (x_{2},x_{3})\rho (x_{1},x_{4})+\rho (x_{3},x_{1})\rho (x_{2},x_{4}).
  \end{eqnarray}
\end{defi}

See \cite{Dzhu} for representations of vector product $n$-Lie algebras.
\begin{ex}
Let $(\g,[\cdot,\cdot,\cdot]_{\g})$ be a $3$-Lie algebra.  The linear map $\ad:\wedge^2\g\rightarrow\gl(\g)$ defines a representation
of the $3$-Lie algebra $\g$ on itself, which is called the {\bf adjoint representation} of $\g.$
\end{ex}

 Let $(V;\rho)$ be a representation of a $3$-Lie algebra  $(\g,[\cdot,\cdot,\cdot]_{\g})$. Denote by
$$C_{\Li}^{n}(\g;V):=\Hom (\underbrace{\wedge^{2} \g\otimes \cdots\otimes \wedge^{2}\g}_{(n-1)}\wedge \g,V),\quad(n\geq 1),$$
which is the space of $n$-cochains.
The coboundary operator ${\rm d}:C_{\Li}^{n}(\g;V)\rightarrow C_{\Li}^{n+1}(\g;V)$ is defined by
\begin{eqnarray*}&&
({\rm d}f)(\mathfrak{X}_1,\cdots,\mathfrak{X}_n,x_{n+1})\\
&=&\sum_{1\leq j<k\leq n}(-1)^{j} f(\mathfrak{X}_1,\cdots,\hat{\mathfrak{X}_{j}},\cdots,\mathfrak{X}_{k-1},
[x_j,y_j,x_k]_{\g}\wedge y_k+x_k\wedge[x_j,y_j,y_k]_{\g},
\mathfrak{X}_{k+1},\cdots,\mathfrak{X}_{n},x_{n+1})\\&&
+\sum_{j=1}^{n}(-1)^{j}f(\mathfrak{X}_1,\cdots,\hat{\mathfrak{X}_{j}},\cdots,\mathfrak{X}_{n},
[x_j,y_j,x_{n+1}]_{\g})\\&&
+\sum_{j=1}^{n}(-1)^{j+1}\rho(x_j,y_j)f(\mathfrak{X}_1,\cdots,\hat{\mathfrak{X}_{j}},
\cdots,\mathfrak{X}_{n},x_{n+1})\\&&
+(-1)^{n+1}\Big(\rho(y_n,x_{n+1})f(\mathfrak{X}_1,\cdots,\mathfrak{X}_{n-1},x_n)+\rho(x_{n+1},x_n)f(\mathfrak{X}_1,\cdots,\mathfrak{X}_{n-1},y_n)\Big),
\end{eqnarray*}
for all$~\mathfrak{X}_{i}=x_{i}\wedge y_{i}\in \wedge^{2}\g,~i=1,2,\cdots,n~and~x_{n+1}\in \g.$ It was proved in \cite{Casas,Takhtajan1} that ${\rm d}\circ{\rm d}=0.$ Thus, $(\oplus_{n=1}^{+\infty}C_{\Li}^{n}(\g;V),{\rm d})$ is a cochain complex.

\begin{defi}
The {\bf cohomology} of the $3$-Lie algebra $\g$ with coefficients in $V$ is the cohomology of the cochain complex $(\oplus_{n=1}^{+\infty} C_{\Li}^{n}(\g;V),{\rm d})$. Denote by $\huaZ_{\Li}^{n}(\g;V)$ and $\huaB_{\Li}^{n}(\g;V)$
the set of $n$-cocycles and the set of $n$-coboundaries, respectively. The $n$-th cohomology group is defined by
\begin{eqnarray*}
\huaH_{\Li}^{n}(\g;V)=\huaZ_{\Li}^{n}(\g;V)/\huaB_{\Li}^{n}(\g;V).
\end{eqnarray*}
\end{defi}

A $2$-cochain $\Phi\in C_{\Li}^{2}(\g;V)=\Hom(\wedge^3\g,V)$ is a {\bf 2-cocycle} on $\g$ with coefficients in $(V;\rho)$ if $\Phi$ satisfies
\begin{eqnarray}\label{2-cocycle}
&&\Phi(x_1,x_2,[x_3,x_4,x_5]_{\g})-\Phi([x_1,x_2,x_3]_{\g},x_4,x_5)-\Phi(x_3,[x_1,x_2,x_4]_{\g},x_5)-\Phi(x_3,x_4,[x_1,x_2,x_5]_{\g})\\
\nonumber&&+\rho(x_1,x_2)\Phi(x_3,x_4,x_5)-\rho(x_3,x_4)\Phi(x_1,x_2,x_5)-\rho(x_4,x_5)\Phi(x_1,x_2,x_3)-\rho(x_5,x_3)\Phi(x_1,x_2,x_4)\\
\nonumber&=&0,
\end{eqnarray}
for all $x_i\in \g, 1\leq i\leq 5$.

Let $\g$ be a $3$-Lie algebra and $(V;\rho)$ be a representation of $\g.$ For any $2$-cocycle $\Phi\in C_{\Li}^{2}(\g;V),$
there is a $3$-Lie algebra structure on the direct sum $\g\oplus V$ of
vector spaces, defined by
\begin{eqnarray}\label{twisted-semi-direct-product}
[(x,u),(y,v),(z,w)]_{\Phi}=([x,y,z]_{\g},\rho(x,y)w+\rho(y,z)u+\rho(z,x)v+\Phi(x,y,z)),
\end{eqnarray}
for $x,y,z\in \g, u,v,w\in V.$ This $3$-Lie algebra is called the {\bf $\Phi$-twisted semi-direct product $3$-Lie algebra} and denoted by $\g\ltimes_{\Phi} V.$

In the sequel, $\Phi$ will always denote a $2$-cocycle.

\begin{defi}
Let $\g$ be a $3$-Lie algebra and $(V;\rho)$ be a representation of $\g.$  A linear operator $T:V\rightarrow\g$ is called a {\bf $\Phi$-twisted Rota-Baxter operator} on $\g$ with respect to $(V;\rho)$ if $T$ satisfies
\begin{equation}
\label{twisted Rota-Baxter operator}[Tu,Tv,Tw]_{\g}=T\Big(\rho(Tu,Tv)w+\rho(Tv,Tw)u+\rho(Tw,Tu)v+\Phi(Tu,Tv,Tw)\Big), \quad \forall~u,v,w\in V.
\end{equation}
\end{defi}

\begin{rmk}
In \cite{BaiRGuo}, the authors introduced the notion of a Rota-Baxter operator of weight $\lambda$ on a $3$-Lie algebra and studied its relation with Rota-Baxter Lie algebras.  In \cite{BGS-3-Bialgebras}, the authors introduced the notion of an $\huaO$-operator on a $3$-Lie algebra with respect to a representation to study solutions of $3$-Lie Yang-Baxter equations. It is straightforward to see that $\Phi$-twisted Rota-Baxter operators on $3$-Lie algebras are generalizations of Rota-Baxter operators of weight $0$ and $\huaO$-operators on  $3$-Lie algebras.

\end{rmk}
We give a construction of $\Phi$-twisted Rota-Baxter operators on $3$-Lie algebras.
\begin{pro}
Let $(\g,[\cdot,\cdot,\cdot]_{\g})$ be a $3$-Lie algebra and $(V;\rho)$ be a representation. Suppose that $f:\g\rightarrow V$ is an invertible linear map. Then $T=f^{-1}:V\rightarrow\g$ is a $\Phi$-twisted Rota-Baxter operator with $\Phi=-{\rm d}f.$ \
\end{pro}
\begin{proof}
By a direct computation, we have
\begin{eqnarray}
\label{cochain-complex}\Phi(Tu,Tv,Tw)&=&-({\rm d}f)(Tu,Tv,Tw)\\
\nonumber              &=&-\rho(Tu,Tv)w-\rho(Tv,Tw)u-\rho(Tw,Tu)v+f([Tu,Tv,Tw]_{\g}).
\end{eqnarray}
By applying $T$ to both sides of \eqref{cochain-complex}, we obtain the identity \eqref{twisted Rota-Baxter operator}.
\end{proof}

The identity \eqref{twisted Rota-Baxter operator} can be characterized by the graph of $T$ being a subalgebra.
\begin{thm}\label{semi-direct}
A linear map $T:V\rightarrow\g$ is a $\Phi$-twisted Rota-Baxter operator on a $3$-Lie algebra $\g$ with respect to a representation $(V;\rho)$ if and only if the graph $Gr(T)=\{(Tu,u)|u\in V\}$ is a subalgebra of the $\Phi$-twisted semi-direct product $3$-Lie algebra $\g\ltimes_{\Phi} V.$
\end{thm}

\begin{proof}
Let $T:V\rightarrow\g$ be a linear map.
For all $u,v,w\in V,$  we have
\begin{eqnarray*}
&&[(Tu,u),(Tv,v),(Tw,w)]_{\Phi}\\
&=&([Tu,Tv,Tw]_{\g},\rho(Tu,Tv)w+\rho(Tv,Tw)u+\rho(Tw,Tu)v+\Phi(Tu,Tv,Tw)),
\end{eqnarray*}
which implies that the graph $Gr(T)=\{(Tu,u)|u\in V\}$ is a subalgebra of the
$\Phi$-twisted semi-direct product $3$-Lie algebra $\g\ltimes_{\Phi} V$ if and only if $T$ satisfies
\begin{eqnarray*}
[Tu,Tv,Tw]_{\g}=T\Big(\rho(Tu,Tv)w+\rho(Tv,Tw)u+\rho(Tw,Tu)v+\Phi(Tu,Tv,Tw)\Big),
\end{eqnarray*}
which means that $T$ is a $\Phi$-twisted Rota-Baxter operator.
\end{proof}

Since $V$ and $Gr(T)$ are isomorphic as vector spaces, we get the following result immediately.
\begin{cor}\label{Induce-3-Lie}
Let $T:V\rightarrow \g$ be a $\Phi$-twisted Rota-Baxter operator on a $3$-Lie algebra $\g$ with respect to a representation
$(V;\rho).$ Then there is a $3$-Lie algebra structure $[\cdot,\cdot,\cdot]_T$ on $V$  given by
\begin{eqnarray}\label{3-Lie-on-v}
\qquad[u,v,w]_{T}=\rho(Tu,Tv)w+\rho(Tv,Tw)u+\rho(Tw,Tu)v+\Phi(Tu,Tv,Tw), \quad \forall u,v,w\in V.
\end{eqnarray}
Furthermore, $T$ is a homomorphism from the 3-Lie algebra $(V,[\cdot,\cdot,\cdot]_T)$ to   $(\g,[\cdot,\cdot,\cdot]_\g)$.
\end{cor}
At the end of this subsection, we introduce the notion of a $T$-admissible $1$-cocycle  by which we construct a new
$\Phi$-twisted Rota-Baxter operator on a 3-Lie algebra.

Let $f:\g\rightarrow V$ be a linear map. Define $\Psi_f:\g\oplus V\rightarrow \g\oplus V$ by $\Psi_f=\left(\begin{array}{cc}
 \Id&0\\
  f&\Id\\
 \end{array}\right).$
\begin{pro}\label{isomorphism}
Let $\g$ be a $3$-Lie algebra and $(V;\rho)$ be a representation. Then $\Psi_f$ is an isomorphism from the $\Phi$-twisted semi-direct product $3$-Lie algebra $\g\ltimes_{\Phi} V$ to the $(\Phi-{\rm d}f)$-twisted semi-direct product $3$-Lie algebra $\g\ltimes_{(\Phi-{\rm d}f)}V$.
\end{pro}
\begin{proof}
For all $(x,u),(y,v),(z,w)\in \g\oplus V,$ we have,
\begin{eqnarray*}
\Psi_{f}[(x,u),(y,v),(z,w)]_{\Phi}&=&\Psi_{f}([x,y,z]_{\g},\rho(x,y)w+\rho(z,x)v+\rho(y,z)u+\Phi(x,y,z))\\
&=&([x,y,z]_{\g},\rho(x,y)w+\rho(z,x)v+\rho(y,z)u+\Phi(x,y,z)+f[x,y,z]_{\g})\\
&=&([x,y,z]_{\g},\rho(x,y)w+\rho(z,x)v+\rho(y,z)u+\Phi(x,y,z)\\
&&+\rho(x,y)f(z)+\rho(y,z)f(x)+\rho(z,x)f(y)-({\rm d} f)(x,y,z))\\
&=&[(x,u+f(x)),(y,v+f(y)),(z,w+f(z))]_{(\Phi-{\rm d} f)}\\
&=&[\Psi_{f}(x,u),\Psi_{f}(y,v),\Psi_{f}(z,w)]_{(\Phi-{\rm d} f)},
\end{eqnarray*}
which implies that $\Psi_f$ is an isomorphism between $3$-Lie algebras.
 \end{proof}

\begin{defi}
Let $T$ be a $\Phi$-twisted Rota-Baxter operator on a $3$-Lie algebra $\g$ with respect to $(V;\rho)$.
A $1$-cocycle $f$ of $\g$  with coefficients in $V$ is called a {\bf$T$-admissible $1$-cocycle} if the linear map $\Id+f\circ T:V\rightarrow V$ is invertible.
\end{defi}

\begin{pro}
With the above notations, if $f:\g\rightarrow V$ is a $T$-admissible $1$-cocycle, then $T(\Id+f\circ T)^{-1}:V\rightarrow \g$ is a $\Phi$-twisted Rota-Baxter operator.
We denote this $\Phi$-twisted Rota-Baxter operator by $T_{f}.$
\end{pro}
\begin{proof}
Let $f$ be a $T$-admissible $1$-cocycle. Then $\Phi-{\rm d} f=\Phi$. By Theorem \ref{semi-direct} and Proposition \ref{isomorphism},
$\Psi_{f}(Gr(T))=\{(Tu,u+f(Tu))|u\in V\}\subset \g\ltimes_{\Phi }V$ is a subalgebra of the $\Phi$-twisted semidirect product $3$-Lie algebra $ \g\ltimes_{\Phi }V$. Since the linear map $(\Id+f\circ T):V\rightarrow V$ is invertible,   $\Psi_{f}(Gr(T))$ is the graph of  $T(\Id+f\circ T)^{-1}:V\rightarrow\g,$
which implies that $T(\Id+f\circ T)^{-1}$ is a $\Phi$-twisted Rota-Baxter operator. This completes the proof.
\end{proof}

Recall from Corollary \ref{Induce-3-Lie} that a $\Phi$-twisted Rota-Baxter operator induces a $3$-Lie algebra on $V.$ We have the following conclusion.
\begin{pro}
Let $T$ be a $\Phi$-twisted Rota-Baxter operator and $f$ be a $T$-admissible $1$-cocycle. Then the $3$-Lie algebra structures on $V$ induced by $T$ and $T_{f}$ are isomorphic.
\end{pro}
\begin{proof}
Consider the  linear map $\Id+f\circ T:V\rightarrow V.$ For all $u,v,w\in V,$ we have
\begin{eqnarray*}
&&[(\Id+f\circ T)(u),(\Id+f\circ T)(v),(\Id+f\circ T)(w)]_{T_{f}}\\
&=&\rho(Tu,Tv)(\Id+f\circ T)(w)+\rho(Tv,Tw)(\Id+f\circ T)(u)\\
&&+\rho(Tw,Tu)(\Id+f\circ T)(v)+\Phi(Tu,Tv,Tw)\\
&=&\rho(Tu,Tv)w+\rho(Tv,Tw)u+\rho(Tw,Tu)v+\Phi(Tu,Tv,Tw)\\
&&+\rho(Tu,Tv)f(Tw)+\rho(Tv,Tw)f(Tu)+\rho(Tw,Tu)f(Tv)\\
&=&[u,v,w]_{T}+f([Tu,Tv,Tw]_{\g})\\
&=&[u,v,w]_{T}+fT([u,v,w]_{T})\\
&=&(\Id+f\circ T)([u,v,w]_{T}).
\end{eqnarray*}
Thus $\Id+f\circ T$ is an isomorphism of $3$-Lie algebras from $(V,[\cdot,\cdot,\cdot]_{T})$ to $(V,[\cdot,\cdot,\cdot]_{T_{f}})$ .
\end{proof}

\subsection{Cohomology  of $\Phi$-twisted Rota-Baxter operators  on $3$-Lie algebras}

In this subsection, we construct a representation of the $3$-Lie algebra $(V,[\cdot,\cdot,\cdot]_T)$ on the vector space $\g,$ and define the cohomology of $\Phi$-twisted Rota-Baxter operators on 3-Lie algebras.
\begin{lem}\label{twisted Rota-Baxter operator-representation}
Let $T$ be a $\Phi$-twisted Rota-Baxter operator on a $3$-Lie algebra $(\g,[\cdot,\cdot,\cdot]_{\g})$ with respect to a representation $(V;\rho)$. Define $\varrho: \wedge^2V\rightarrow\gl(\g)$ by
 \begin{equation}
 \quad\varrho(u,v)(x)=[Tu,Tv,x]_{\g}-T\Big(\rho(x,Tu)v+\rho(Tv,x)u+\Phi(Tu,Tv,x)\Big),\quad \forall x\in \g,u,v\in V.
  \end{equation}
Then $(\g;\varrho)$ is a representation of the $3$-Lie algebra $(V,[\cdot,\cdot,\cdot]_T)$.
\end{lem}
\begin{proof}
By a direct calculation using \eqref{eq:jacobi1}, \eqref{representation-1}-\eqref{2-cocycle}, \eqref{twisted Rota-Baxter operator} and \eqref{3-Lie-on-v},
for all $u_i\in V,1\leq i\leq 4, x\in \g,$ we have
\begin{eqnarray*}
&&\Big(\varrho(u_1,u_2)\varrho(u_3,u_4)-\varrho(u_3,u_4)\varrho(u_1,u_2)-\varrho([u_1,u_2,u_3]_{T},u_4)+\varrho([u_1,u_2,u_4]_{T},u_3)\Big)(x)\\
&=&\varrho(u_1,u_2)\Big([Tu_3,Tu_4,x]_{\g}-T\rho(x,Tu_3)u_4-T\rho(Tu_4,x)u_3-T\Phi(Tu_3,Tu_4,x)\Big)\\
&&-\varrho(u_3,u_4)\Big([Tu_1,Tu_2,x]_{\g}-T\rho(x,Tu_1)u_2-T\rho(Tu_2,x)u_1-T\Phi(Tu_1,Tu_2,x)\Big)\\
&&+T\Big(\rho(x,T[u_1,u_2,u_3]_{T})u_4+\rho(Tu_4,x)[u_1,u_2,u_3]_{T}+\Phi(T[u_1,u_2,u_3]_{T},Tu_4,x)\Big)\\
&&-T\Big(\rho(x,T[u_1,u_2,u_4]_{T})u_3+\rho(Tu_3,x)[u_1,u_2,u_4]_{T}+\Phi(T[u_1,u_2,u_4]_{T},Tu_3,x)\Big)\\
&&-[T[u_1,u_2,u_3]_{T},Tu_4,x]_{\g}+[T[u_1,u_2,u_4]_{T},Tu_3,x]_{\g}\\
&=&-[Tu_1,Tu_2,T\rho(x,Tu_3)u_4]_{\g}-[Tu_1,Tu_2,T\rho(Tu_4,x)u_3]_{\g}-[Tu_1,Tu_2,T\Phi(Tu_3,Tu_4,x)]_{\g}\\
&&-T\Big(\rho([Tu_3,Tu_4,x]_{\g},Tu_1)u_2+\rho(Tu_2,[Tu_3,Tu_4,x]_{\g})u_1+\Phi(Tu_1,Tu_2,[Tu_3,Tu_4,x]_{\g})\Big)\\
&&+T\Big(\rho(T\rho(x,Tu_3)u_4,Tu_1)u_2+\rho(Tu_2,T\rho(x,Tu_3)u_4)u_1+\Phi(Tu_1,Tu_2,T\rho(x,Tu_3)u_4)\Big)\\
&&+T\Big(\rho(T\rho(Tu_4,x)u_3,Tu_1)u_2+\rho(Tu_2,T\rho(Tu_4,x)u_3)u_1+\Phi(Tu_1,Tu_2,T\rho(Tu_4,x)u_3)\Big)\\
&&+T\Big(\rho(T\Phi(Tu_3,Tu_4,x),Tu_1)u_2+\rho(Tu_2,T\Phi(Tu_3,Tu_4,x))u_1+\Phi(Tu_1,Tu_2,T\Phi(Tu_3,Tu_4,x))\Big)\\
&&+[Tu_3,Tu_4,T\rho(x,Tu_1)u_2]_{\g}+[Tu_3,Tu_4,T\rho(Tu_2,x)u_1]_{\g}+[Tu_3,Tu_4,T\Phi(Tu_1,Tu_2,x)]_{\g}\\
&&+T\Big(\rho([Tu_1,Tu_2,x]_{\g},Tu_3)u_4+\rho(Tu_4,[Tu_1,Tu_2,x]_{\g})u_3+\Phi(Tu_3,Tu_4,[Tu_1,Tu_2,x]_{\g})\Big)\\
&&-T\Big(\rho(T\rho(x,Tu_1)u_2,Tu_3)u_4+\rho(Tu_4,T\rho(x,Tu_1)u_2)u_3+\Phi(Tu_3,Tu_4,T\rho(x,Tu_1)u_2)\Big)\\
&&-T\Big(\rho(T\rho(Tu_2,x)u_1,Tu_3)u_4+\rho(Tu_4,T\rho(Tu_2,x)u_1)u_3+\Phi(Tu_3,Tu_4,T\rho(Tu_2,x)u_1)\Big)\\
&&-T\Big(\rho(T\Phi(Tu_1,Tu_2,x),Tu_3)u_4+\rho(Tu_4,T\Phi(Tu_1,Tu_2,x))u_1+\Phi(Tu_3,Tu_4,T\Phi(Tu_1,Tu_2,x))\Big)\\
&&+T\Big(\rho(x,[Tu_1,Tu_2,Tu_3]_{\g})u_4+\rho(Tu_4,x)\rho(Tu_1,Tu_2)u_3+\rho(Tu_4,x)\rho(Tu_2,Tu_3)u_1\\
&&+\rho(Tu_4,x)\rho(Tu_3,Tu_1)u_2+\Phi([Tu_1,Tu_2,Tu_3]_{\g},Tu_4,x)\Big)\\
&&-T\Big(\rho(x,[Tu_1,Tu_2,Tu_4]_{\g})u_3+\rho(Tu_3,x)\rho(Tu_1,Tu_2)u_4+\rho(Tu_3,x)\rho(Tu_2,Tu_4)u_1\\
&&+\rho(Tu_3,x)\rho(Tu_4,Tu_1)u_2+\Phi([Tu_1,Tu_2,Tu_4]_{\g},Tu_3,x)\Big)\\
&=&0,
\end{eqnarray*}
and
\begin{eqnarray*}
&&\Big(\varrho([u_1,u_2,u_3]_{T},u_4)-\varrho(u_1,u_2)\varrho(u_3,u_4)-\varrho(u_2,u_3)\varrho(u_1,u_4)-\varrho(u_3,u_1)\varrho(u_2,u_4)\Big)(x)\\
&=&-T\Big(\rho(x,T[u_1,u_2,u_3]_{T})u_4+\rho(Tu_4,x)[u_1,u_2,u_3]_{T}+\Phi(T[u_1,u_2,u_3]_{T},Tu_4,x)\Big)\\
&&-\varrho(u_1,u_2)\Big([Tu_3,Tu_4,x]_{\g}-T\rho(x,Tu_3)u_4-T\rho(Tu_4,x)u_3-T\Phi(Tu_3,Tu_4,x)\Big)\\
&&-\varrho(u_2,u_3)\Big([Tu_1,Tu_4,x]_{\g}-T\rho(x,Tu_1)u_4-T\rho(Tu_4,x)u_1-T\Phi(Tu_1,Tu_4,x)\Big)\\
&&-\varrho(u_3,u_1)\Big([Tu_2,Tu_4,x]_{\g}-T\rho(x,Tu_2)u_4-T\rho(Tu_4,x)u_2-T\Phi(Tu_2,Tu_4,x)\Big)\\
&&+[T[u_1,u_2,u_3]_{T},Tu_4,x]_{\g}\\
&=&-T\Big(\rho(x,[Tu_1,Tu_2,Tu_3]_{\g})u_4+\Phi([Tu_1,Tu_2,Tu_3]_{\g},Tu_4,x)\Big)\\
&&-T\rho(Tu_4,x)\Big(\rho(Tu_1,Tu_2)u_3+\rho(Tu_2,Tu_3)u_1+\rho(Tu_3,Tu_1)u_2+\Phi(Tu_1,Tu_2,Tu_3)\Big)\\
&&+[Tu_1,Tu_2,T\rho(x,Tu_3)u_4+T\rho(Tu_4,x)u_3+T\Phi(Tu_3,Tu_4,x)]_{\g}\\
&&+T\Big(\rho([Tu_3,Tu_4,x]_{\g},Tu_1)u_2+\rho(Tu_2,[Tu_3,Tu_4,x]_{\g})u_1+\Phi(Tu_1,Tu_2,[Tu_3,Tu_4,x]_{\g})\Big)\\
&&-T\Big(\rho(T\rho(x,Tu_3)u_4,Tu_1)u_2+\rho(Tu_2,T\rho(x,Tu_3)u_4)u_1+\Phi(Tu_1,Tu_2,T\rho(x,Tu_3)u_4)\Big)\\
&&-T\Big(\rho(T\rho(Tu_4,x)u_3,Tu_1)u_2+\rho(Tu_2,T\rho(Tu_4,x)u_3)u_1+\Phi(Tu_1,Tu_2,T\rho(Tu_4,x)u_3)\Big)\\
&&-T\Big(\rho(T\Phi(Tu_3,Tu_4,x),Tu_1)u_2+\rho(Tu_2,T\Phi(Tu_3,Tu_4,x))u_1+\Phi(Tu_1,Tu_2,T\Phi(Tu_3,Tu_4,x))\Big)\\
&&+[Tu_2,Tu_3,T\rho(x,Tu_1)u_4+T\rho(Tu_4,x)u_1+T\Phi(Tu_1,Tu_4,x)]_{\g}\\
&&+T\Big(\rho([Tu_1,Tu_4,x]_{\g},Tu_2)u_3+\rho(Tu_3,[Tu_1,Tu_4,x]_{\g})u_2+\Phi(Tu_2,Tu_3,[Tu_1,Tu_4,x]_{\g})\Big)\\
&&-T\Big(\rho(T\rho(x,Tu_1)u_4,Tu_2)u_3+\rho(Tu_3,T\rho(x,Tu_1)u_4)u_2+\Phi(Tu_2,Tu_3,T\rho(x,Tu_1)u_4)\Big)\\
&&-T\Big(\rho(T\rho(Tu_4,x)u_1,Tu_2)u_3+\rho(Tu_3,T\rho(Tu_4,x)u_1)u_2+\Phi(Tu_2,Tu_3,T\rho(Tu_4,x)u_1)\Big)\\
&&-T\Big(\rho(T\Phi(Tu_1,Tu_4,x),Tu_2)u_3+\rho(Tu_3,T\Phi(Tu_1,Tu_4,x))u_2+\Phi(Tu_2,Tu_3,T\Phi(Tu_1,Tu_4,x))\Big)\\
&&+[Tu_3,Tu_1,T\rho(x,Tu_2)u_4+T\rho(Tu_4,x)u_2+T\Phi(Tu_2,Tu_4,x)]_{\g}\\
&&+T\Big(\rho([Tu_2,Tu_4,x]_{\g},Tu_3)u_1+\rho(Tu_1,[Tu_2,Tu_4,x]_{\g})u_3+\Phi(Tu_3,Tu_1,[Tu_2,Tu_4,x]_{\g})\Big)\\
&&-T\Big(\rho(T\rho(x,Tu_2)u_4,Tu_3)u_1+\rho(Tu_1,T\rho(x,Tu_2)u_4)u_3+\Phi(Tu_3,Tu_1,T\rho(x,Tu_2)u_4)\Big)\\
&&-T\Big(\rho(T\rho(Tu_4,x)u_2,Tu_3)u_1+\rho(Tu_1,T\rho(Tu_4,x)u_2)u_3+\Phi(Tu_3,Tu_1,T\rho(Tu_4,x)u_2)\Big)\\
&&-T\Big(\rho(T\Phi(Tu_2,Tu_4,x),Tu_3)u_1+\rho(Tu_1,T\Phi(Tu_2,Tu_4,x))u_3+\Phi(Tu_3,Tu_1,T\Phi(Tu_2,Tu_4,x))\Big)\\
&=&0.
\end{eqnarray*}
Therefore, we deduce that $(\g;\varrho)$ is a representation of the $3$-Lie algebra $(V,[\cdot,\cdot,\cdot]_{T})$.
\end{proof}

Let ${\rm d_{T}}:C_{\Li}^{n}(V;\g)\rightarrow C_{\Li}^{n+1}(V;\g),(n\geq1)$ be the corresponding coboundary
operator of the $3$-Lie algebra $(V,[\cdot,\cdot,\cdot]_{T})$ with coefficients in the representation $(\g;\varrho)$.
More precisely, ${\rm d_{T}}:C_{\Li}^{n}(V;\g)\rightarrow C_{\Li}^{n+1}(V;\g)$ is given by
\begin{eqnarray*}
&&({\rm d_{T}}f)(\U_1,\cdots,\U_n,u_{n+1})\\
&=&\sum_{1\leq j<k\leq n}(-1)^jf(\U_1,\cdots,\hat{\U_j},\cdots,\U_{k-1},[u_j,v_j,u_k]_{T}\wedge v_k+u_k\wedge[u_j,v_j,v_k]_{T},\U_{k+1},\cdots,\U_n,u_{n+1})\\
&&+\sum_{j=1}^{n}(-1)^{j}f(\U_1,\cdots,\hat{\U_j},\cdots,\U_{n},[u_j,v_j,u_{n+1}]_{T})\\
&&+\sum_{j=1}^{n}(-1)^{j+1}\varrho(u_j,v_j)f(\U_1,\cdots,\hat{\U_j},\cdots,\U_n,u_{n+1})\\
&&+(-1)^{n+1}\Big(\varrho(v_n,u_{n+1})f(\U_1,\cdots,\U_{n-1},u_n)+\varrho(u_{n+1},u_{n})f(\U_1,\cdots,\U_{n+1},v_n)\Big),
\end{eqnarray*}
for all $\U_i=u_i\wedge v_i\in \wedge^2V,~ i=1,2,\cdots,n$ and $u_{n+1}\in V.$

It is obvious that $f\in C_{\Li}^{1}(V;\g)$ is closed if and only if
\begin{eqnarray*}
&&[Tu_1,Tu_2,f(u_3)]_{\g}+[f(u_1),Tu_2,Tu_3]_{\g}+[Tu_1,f(u_2),Tu_3]_{\g}\\
&=&T\Big(\rho(f(u_3),Tu_1)u_2+\rho(Tu_2,f(u_3))u_1+\Phi(Tu_1,Tu_2,f(u_3))\Big)\\
&&+T\Big(\rho(f(u_1),Tu_2)u_3+\rho(Tu_3,f(u_1))u_2+\Phi(Tu_2,Tu_3,f(u_1))\Big)\\
&&+T\Big(\rho(f(u_2),Tu_3)u_1+\rho(Tu_1,f(u_2))u_3+\Phi(Tu_3,Tu_1,f(u_2))\Big)\\
&&+f\Big(\rho(Tu_1,Tu_2)u_3+\rho(Tu_2,Tu_3)u_1+\rho(Tu_3,Tu_1)u_2+\Phi(Tu_1,Tu_2,Tu_3)\Big).
\end{eqnarray*}
Define $\delta:\wedge^2\g\rightarrow\Hom(V,\g)$ by
\begin{eqnarray*}
\delta(\mathfrak{X})v=T\rho(\mathfrak{X})v-[\mathfrak{X},Tv]_{\g}+T\Phi(\mathfrak{X},Tv), \quad \forall\mathfrak{X}\in\wedge^2\g, v\in V.
\end{eqnarray*}
\begin{pro}
Let $T$ be a $\Phi$-twisted Rota-Baxter operator on a $3$-Lie algebra $(\g,[\cdot,\cdot,\cdot]_{\g})$
with respect to a representation $(V;\rho).$ Then $\delta(\mathfrak{X})$ is a $1$-cocycle of the $3$-Lie algebra $(V,[\cdot,\cdot,\cdot]_{T})$ with coefficients in $(\g;\varrho).$
\end{pro}
\begin{proof}
For all $u_1,u_2,u_3\in V,$ we have
\begin{eqnarray*}
&&({\rm d_{T}}\delta(\mathfrak{X}))(u_1,u_2,u_3)\\
&=&[Tu_1,Tu_2,T\rho(\mathfrak{X})u_3]_{\g}-[Tu_1,Tu_2,[\mathfrak{X},Tu_3]_{\g}]_{\g}+[Tu_1,Tu_2,T\Phi(\mathfrak{X},Tu_3)]_{\g}\\
&&+[T\rho(\mathfrak{X})u_1,Tu_2,Tu_3]_{\g}-[[\mathfrak{X},Tu_1]_{\g},Tu_2,Tu_3,]_{\g}+[T\Phi(\mathfrak{X},Tu_1),Tu_2,Tu_3]_{\g}\\
&&+[Tu_1,T\rho(\mathfrak{X})u_2,Tu_3]_{\g}-[Tu_1,[\mathfrak{X},Tu_2]_{\g},Tu_3]_{\g}+[Tu_1,T\Phi(\mathfrak{X},Tu_2),Tu_3]_{\g}\\
&&-T\rho(\mathfrak{X})\Big(\rho(Tu_1,Tu_2)u_3+\rho(Tu_2,Tu_3)u_1+\rho(Tu_3,Tu_1)u_2+\Phi(Tu_1,Tu_2,Tu_3)\Big)\\
&&+[\mathfrak{X},[Tu_1,Tu_2,Tu_3]_{\g}]_{\g}-T\Phi(\mathfrak{X},[Tu_1,Tu_2,Tu_3]_{\g})\\
&&-T\rho\Big(T\rho(\mathfrak{X})u_3-[\mathfrak{X},Tu_3]_{\g}+T\Phi(\mathfrak{X},Tu_3),Tu_1\Big)u_2\\
&&-T\rho\Big(Tu_2,T\rho(\mathfrak{X})u_3-[\mathfrak{X},Tu_3]_{\g}+T\Phi(\mathfrak{X},Tu_3)\Big)u_1\\
&&-T\Phi\Big(Tu_1,Tu_2,T\rho(\mathfrak{X})u_3-[\mathfrak{X},Tu_3]_{\g}+T\Phi(\mathfrak{X},Tu_3)\Big)\\
&&-T\rho\Big(T\rho(\mathfrak{X})u_1-[\mathfrak{X},Tu_1]_{\g}+T\Phi(\mathfrak{X},Tu_1),Tu_2\Big)u_3\\
&&-T\rho\Big(Tu_3,T\rho(\mathfrak{X})u_1-[\mathfrak{X},Tu_1]_{\g}+T\Phi(\mathfrak{X},Tu_1)\Big)u_2\\
&&-T\Phi\Big(Tu_2,Tu_3,T\rho(\mathfrak{X})u_1-[\mathfrak{X},Tu_1]_{\g}+T\Phi(\mathfrak{X},Tu_1)\Big)\\
&&-T\rho\Big(T\rho(\mathfrak{X})u_2-[\mathfrak{X},Tu_2]_{\g}+T\Phi(\mathfrak{X},Tu_2),Tu_3\Big)u_1\\
&&-T\rho\Big(Tu_1,T\rho(\mathfrak{X})u_2-[\mathfrak{X},Tu_2]_{\g}+T\Phi(\mathfrak{X},Tu_2)\Big)u_3\\
&&-T\Phi\Big(Tu_3,Tu_1,T\rho(\mathfrak{X})u_2-[\mathfrak{X},Tu_2]_{\g}+T\Phi(\mathfrak{X},Tu_2)\Big)\\
&=&T\Big(\rho(Tu_1,Tu_2)\rho(\mathfrak{X})u_3+\rho(Tu_2,T\rho(\mathfrak{X})u_3)u_1+\rho(T\rho(\mathfrak{X})u_3,Tu_1)u_2+\Phi(Tu_1,Tu_2,T\rho(\mathfrak{X})u_3)\Big)\\
&&+T\Big(\rho(T\rho(\mathfrak{X})u_1,Tu_2)u_3+\rho(Tu_2,Tu_3)\rho(\mathfrak{X})u_1+\rho(Tu_3,T\rho(\mathfrak{X})u_1)u_2+\Phi(T\rho(\mathfrak{X})u_1,Tu_2,Tu_3)\Big)\\
&&+T\Big(\rho(Tu_1,T\rho(\mathfrak{X})u_2)u_3+\rho(T\rho(\mathfrak{X})u_2,Tu_3)u_1+\rho(Tu_3,Tu_1)\rho(\mathfrak{X})u_2+\Phi(Tu_1,T\rho(\mathfrak{X})u_2,,Tu_3)\Big)\\
&&+[Tu_1,Tu_2,T\Phi(\mathfrak{X},Tu_3)]_{\g}+[T\Phi(\mathfrak{X},Tu_1),Tu_2,Tu_3]_{\g}+[Tu_1,T\Phi(\mathfrak{X},Tu_2),Tu_3]_{\g}\\
&&-T\rho(\mathfrak{X})\Big(\rho(Tu_1,Tu_2)u_3+\rho(Tu_2,Tu_3)u_1+\rho(Tu_3,Tu_1)u_2+\Phi(Tu_1,Tu_2,Tu_3)\Big)\\
&&-T\Big(\rho(T\rho(\mathfrak{X})u_3,Tu_1)u_2-\rho([\mathfrak{X},Tu_3]_{\g},Tu_1)u_2+\rho(T\Phi(\mathfrak{X},Tu_3),Tu_1)u_2\Big)\\
&&-T\Big(\rho(Tu_2,T\rho(\mathfrak{X})u_3)u_1-\rho(Tu_2,[\mathfrak{X},Tu_3]_{\g})u_1+\rho(Tu_2, T\Phi(\mathfrak{X},Tu_3)u_1)\Big)\\
&&-T\Big(\Phi(Tu_1,Tu_2,T\rho(\mathfrak{X})u_3)-\Phi(Tu_1,Tu_2,[\mathfrak{X},Tu_3]_{\g})+\Phi(Tu_1,Tu_2,T\Phi(\mathfrak{X},Tu_3))\Big)\\
&&-T\Big(\rho(T\rho(\mathfrak{X})u_1,Tu_2)u_3-\rho([\mathfrak{X},Tu_1]_{\g},Tu_2)u_3+\rho(T\Phi(\mathfrak{X},Tu_1),Tu_2)u_3\Big)\\
&&-T\Big(\rho(Tu_3,T\rho(\mathfrak{X})u_1)u_2-\rho(Tu_3,[\mathfrak{X},Tu_1]_{\g})u_2+\rho(Tu_3, T\Phi(\mathfrak{X},Tu_1)u_2)\Big)\\
&&-T\Big(\Phi(Tu_2,Tu_3,T\rho(\mathfrak{X})u_1)-\Phi(Tu_2,Tu_3,[\mathfrak{X},Tu_1]_{\g})+\Phi(Tu_2,Tu_3,T\Phi(\mathfrak{X},Tu_1))\Big)\\
&&-T\Big(\rho(T\rho(\mathfrak{X})u_2,Tu_3)u_1-\rho([\mathfrak{X},Tu_2]_{\g},Tu_3)u_1+\rho(T\Phi(\mathfrak{X},Tu_2),Tu_3)u_1\Big)\\
&&-T\Big(\rho(Tu_1,T\rho(\mathfrak{X})u_2)u_3-\rho(Tu_1,[\mathfrak{X},Tu_2]_{\g})u_3+\rho(Tu_1, T\Phi(\mathfrak{X},Tu_2)u_3)\Big)\\
&&-T\Big(\Phi(Tu_3,Tu_1,T\rho(\mathfrak{X})u_2)-\Phi(Tu_3,Tu_1,[\mathfrak{X},Tu_2]_{\g})+\Phi(Tu_3,Tu_1,T\Phi(\mathfrak{X},Tu_2))\Big)\\
&&-T\Phi(\mathfrak{X},[Tu_1,Tu_2,Tu_3]_{\g})\\
&=&T\Big(\Phi(Tu_1,Tu_2,[\mathfrak{X},Tu_3]_{\g})+\Phi(Tu_2,Tu_3,[\mathfrak{X},Tu_1]_{\g})+\Phi(Tu_3,Tu_1,[\mathfrak{X},Tu_2]_{\g})\\
&&+\rho(Tu_1,Tu_2)\Phi(\mathfrak{X},Tu_3)+\rho(Tu_2,Tu_3)\Phi(\mathfrak{X},Tu_1)+\rho(Tu_3,Tu_1)\Phi(\mathfrak{X},Tu_2)\\
&&-\Phi(\mathfrak{X},[Tu_1,Tu_2,Tu_3]_{\g})-\rho(\mathfrak{X})\Phi(Tu_1,Tu_2,Tu_3)\Big)\\
&=&0.
\end{eqnarray*}
Thus, we deduce that ${\rm d_{T}}\delta(\mathfrak{X})=0.$ The proof is finished.
\end{proof}
We now give the cohomology of $\Phi$-twisted Rota-Baxter operators on $3$-Lie algebras.

Let $T$ be a $\Phi$-twisted Rota-Baxter operator on a $3$-Lie algebra $(\g,[\cdot,\cdot,\cdot]_{\g})$
with respect to a representation $(V;\rho).$ Define the set of $n$-cochains by
\begin{eqnarray}
C_{T}^{n}(V;\g)=
\left\{\begin{array}{rcl}
{}C_{\Li}^{n-1}(V;\g),\quad n\geq 2,\\
{}\g\wedge\g,\quad n=1.
\end{array}\right.
\end{eqnarray}

Define ${\partial}:C_{T}^{n}(V;\g)\rightarrow C_{T}^{n+1}(V;\g)$ by
\begin{eqnarray}
{\partial}=
\left\{\begin{array}{rcl}
{}{\rm d_{T}},\quad n\geq 2,\\
{}\delta,\quad n=1.
\end{array}\right.
\end{eqnarray}
Then $(\mathop{\oplus}\limits_{n=1}^{\infty} C_{T}^{n}(V;\g),\partial)$ is a cochain complex. Denote the set of $n$-cocycles by $\huaZ^n_T(V;\g),$ the set of $n$-coboundaries by $\huaB^n_T(V;\g)$ and $n$-th cohomology group by
\begin{eqnarray}
\huaH^n_T(V;\g)=\huaZ^n(V;\g)/\huaB^n(V;\g),\quad n\geq1.
\end{eqnarray}

\begin{defi}
The cohomology of the cochain complex $(\mathop{\oplus}\limits_{n=1}^{\infty} C_{T}^{n}(V;\g),\partial)$  is taken to be the {\bf cohomology for the  $\Phi$-twisted Rota-Baxter operator $T$}.
\end{defi}

\subsection{Infinitesimal deformations of $\Phi$-twisted Rota-Baxter operators  on $3$-Lie algebras}
Now we use the cohomology theory to characterize   infinitesimal deformations of $\Phi$-twisted Rota-Baxter operators.
\begin{defi}
Let $T$ and $T'$ be $\Phi$-twisted Rota-Baxter operators on a $3$-Lie algebra $\g$ with respect to a representation $(V;\rho).$ A {\bf homomorphism} of
$\Phi$-twisted Rota-Baxter operators from $T$ to $T'$ consists a $3$-Lie algebra homomorphism $\phi:\g\rightarrow \g$ and a linear map
$\psi:V\rightarrow V$ such that
\begin{eqnarray}
 \label{condition-1}\phi\circ T&=&T'\circ \psi,\\
  \label{condition-2}\psi(\rho(x,y)u)&=&\rho(\phi(x),\phi(y))(\psi(u)),\quad \forall x,y\in \g, u\in V,\\
  \label{condition-3}\psi\circ \Phi&=&\Phi\circ(\phi\otimes\phi\otimes\phi).
\end{eqnarray}
In particular, if both $\phi$ and $\psi$ are invertible, $(\phi,\psi)$ is called an {\bf isomorphism} from $T$ to $T'$.
\end{defi}

\begin{defi}
Let $T:V\rightarrow \g$ be a $\Phi$-twisted Rota-Baxter operator on a $3$-Lie algebra $(\g,[\cdot,\cdot,\cdot]_{\g})$ with respect to a representation $(V;\rho)$ and
$\frkT:V\rightarrow\g$ a linear map. If $T_t=T+t\frkT$ is still a $\Phi$-twisted Rota-Baxter operator on the $3$-Lie algebra $\g$ with respect to the representation $(V;\rho)$
for all $t$, we say that $\frkT$ generates a {\bf one-parameter infinitesimal deformation} of $T$.
\end{defi}
It is direct to deduce that $T_t=T+t\frkT $ is a one-parameter infinitesimal deformation of the $\Phi$-twisted Rota-Baxter operator $T$ if and only if for any $u,v,w\in V,$
\begin{eqnarray}
\label{equivalent-1}\qquad&&[\frkT u,Tv,Tw]_{\g}+[Tu,\frkT v,Tw]_{\g}+[Tu,Tv,\frkT w]_{\g}\\
\nonumber&=&T\Big(\rho(\frkT w,Tu)v+\rho(Tv,\frkT w)u+\rho(\frkT u,Tv)w+\rho(Tw,\frkT u)v+\rho(\frkT v,Tw)u\\
\nonumber&&+\rho(Tu,\frkT v)w+\Phi(\frkT u,Tv,Tw)+\Phi(Tu,\frkT v,Tw)+\Phi(Tu,Tv,\frkT w)\Big)\\
\nonumber&&+\frkT\Big(\rho(Tu,Tv)w+\rho(Tv,Tw)u+\rho(Tw,Tu)v+\Phi(Tu,Tv,Tw)\Big),\\
\label{equivalent-2}\qquad&&[\frkT u,\frkT v,T(w)]_{\g}+[\frkT v,\frkT w,T(u)]_{\g}+[\frkT w,\frkT u,T(v)]_{\g}\\
\nonumber&=&\frkT\Big(\rho(Tw,\frkT u)v+\rho(\frkT v,Tw)u+\rho(Tu,\frkT v)w+\rho(\frkT w,Tu)v+\rho(Tv,\frkT w)u\\
\nonumber&&+\rho(\frkT u,Tv)w+\Phi(\frkT u,Tv,Tw)+\Phi(Tu,\frkT v,Tw)+\Phi(Tu,Tv,\frkT w)\Big)\\
\nonumber&&+T\Big(\rho(\frkT u,\frkT v)w+\rho(\frkT v,\frkT w)u+\rho(\frkT w,\frkT u)v\\
\nonumber&&+\Phi(Tu,\frkT v,\frkT w)+\Phi(\frkT u,Tv,\frkT w)+\Phi(\frkT u,\frkT v,Tw)\Big),\\
\label{equivalent-3}\qquad&&[\frkT u,\frkT v,\frkT w]_{\g}\\
\nonumber&=&T\Phi(\frkT u,\frkT v,\frkT w)+\frkT\Big(\rho(\frkT u,\frkT v)w+\rho(\frkT v,\frkT w)u+\rho(\frkT w,\frkT u)v\Big)\\
\nonumber&&+\frkT\Big(\Phi(Tu,\frkT v,\frkT w)+\Phi(\frkT u,Tv,\frkT w)+\Phi(\frkT u,\frkT v,Tw)\Big),\\
\label{equivalent-4}\qquad&&\frkT\Phi(\frkT u,\frkT v,\frkT w)=0.
\end{eqnarray}

 Note that \eqref{equivalent-1} means that $\frkT$ is a $1$-cocycle of the $\Phi$-twisted Rota-Baxter operator $T$. Hence, $\frkT$ defines a cohomology class in $\huaH^2_{T}(V;\g)$.

\begin{defi}
Let $T$ be a $\Phi$-twisted Rota-Baxter operator on a $3$-Lie algebra $\g$ with respect to a representation $(V;\rho).$ Two one-parameter infinitesimal deformations $T^1_{t}=T+t\frkT_{1}$ and  $T^2_{t}=T+t\frkT_{2}$ are said to be {\bf equivalent} if there exist $\mathfrak{X}\in\g\wedge\g$ such that $(\Id_{\g}+t\ad_{\mathfrak{X}},\Id_{V}+t\rho(\mathfrak{X})+t\Phi(\mathfrak{X},T))$ is a homomorphism from $T^1_{t}$ to $T^2_{t}$, where $\Phi(\mathfrak{X},T):V \rightarrow V$ is defined by
$$\Phi(\mathfrak{X},T)(u)=\Phi(\mathfrak{X},Tu),\quad \forall u\in V.$$
In particular, a one-parameter infinitesimal deformation $T_{t}=T+t\frkT_{1}$ of a $\Phi$-twisted Rota-Baxter operator $T$ is said to be {\bf trivial} if there exist $\mathfrak{X}\in \g\wedge\g$ such that $(\Id_{\g}+t\ad_{\mathfrak{X}},\Id_{V}+t\rho(\mathfrak{X})+t\Phi(\mathfrak{X},T))$ is a homomorphism from $T_{t}$ to $T.$
\end{defi}

Let $(\Id_{\g}+t\ad_{\mathfrak{X}},\Id_{V}+t\rho(\mathfrak{X})+t\Phi(\mathfrak{X},T))$ be a homomorphism from $T^1_{t}$ to $T^2_{t}.$
By \eqref{condition-1} we get,
\begin{equation*}
\quad (\Id_{\g}+t\ad_{\mathfrak{X}})(T+t\frkT_1)(u)=(T+t\frkT_2)(\Id_V+t\rho(\mathfrak{X})+t\Phi(\mathfrak{X},T))(u),
\end{equation*}
which implies
\begin{eqnarray}\label{Nijenhuis-element-4}
\left\{\begin{array}{rcl}
{}\frkT_1(u)-\frkT_2(u)&=&T\rho(\mathfrak{X})u-[\mathfrak{X},Tu]_{\g}+T\Phi(\mathfrak{X},Tu),\\
{}[\mathfrak{X},\frkT_1(u)]_{\g}&=&\frkT_2(\rho(\mathfrak{X})u+\Phi(\mathfrak{X},Tu)).
\end{array}\right.
\end{eqnarray}

Now we are ready to give the main result in this section.
\begin{thm}
Let $T$ be a $\Phi$-twisted Rota-Baxter operator on a $3$-Lie algebra $(\g,[\cdot,\cdot,\cdot]_{\g})$ with respect to a representation $(V;\rho).$
If two one-parameter infinitesimal deformations $T^1_{t}=T+t\frkT_{1}$ and $T^2_{t}=T+t\frkT_{2}$ are equivalent, then $\frkT_{1}$ and $\frkT_{2}$
 are in the same cohomology class in $\huaH^2_{T}(V;\g)$.
\end{thm}
\begin{proof}
 It is easy to see from the first condition of \eqref{Nijenhuis-element-4} that
\begin{eqnarray*}
 \frkT_1(u)&=& \frkT_2(u)+(\partial\mathfrak{X})(u),\quad \forall u\in V,
\end{eqnarray*}
which implies that $\frkT_1$ and $\frkT_2$ are in the same cohomology class.
\end{proof}

\section{NS-$3$-Lie algebras}\label{sec:GM}
In this section, we introduce the notion of an NS-$3$-Lie algebra. An NS-$3$-Lie algebra gives rise to a 3-Lie algebra and a representation on itself. We show that a twisted Rota-Baxter operator induces an NS-$3$-Lie algebra. Thus NS-$3$-Lie algebras can be viewed as the underlying
algebraic structures of twisted Rota-Baxter operators on $3$-Lie algebras.

\begin{defi}\label{defi-3-NS-Lie algebra}
Let $A$ be a vector space together with two linear maps $\{\cdot,\cdot,\cdot\},[\cdot,\cdot,\cdot]:A\otimes A\otimes A\rightarrow A$ in which $[\cdot,\cdot,\cdot]$ is skew-symmetric. The triple $(A,\{\cdot,\cdot,\cdot\},[\cdot,\cdot,\cdot])$ is called an
{\bf NS-$3$-Lie algebra} if the following identities hold:
\begin{eqnarray}
\label{3-NS-Lie-1}\{x_1,x_2,x_3\}&=&-\{x_2,x_1,x_3\},\\
\label{3-NS-Lie-2}\{x_1,x_2,\{x_3,x_4,x_5\}\}&=&\{x_3,x_4,\{x_1,x_2,x_5\}\}+\{\Courant{x_1,x_2,x_3},x_4,x_5\}\\
\nonumber&&+\{x_3,\Courant{x_1,x_2,x_4},x_5\},\\
\label{3-NS-Lie-3}\{\Courant{x_1,x_2,x_3},x_4,x_5\}&=&\circlearrowleft_{x_1,x_2,x_3}\{x_1,x_2,\{x_3,x_4,x_5\}\},\\
\label{3-NS-Lie-4}[x_1,x_2,\Courant{x_3,x_4,x_5}]&=&\circlearrowleft_{x_3,x_4,x_5}[ x_3,x_4,\Courant{x_1,x_2,x_5}]-\{x_1,x_2,[x_3,x_4,x_5]\}\\
\nonumber&&+\circlearrowleft_{x_3,x_4,x_5}\{x_3,x_4,[x_1,x_2,x_5]\},
\end{eqnarray}
for all $x_i\in A, 1\leq i\leq5$ and $\Courant{\cdot,\cdot,\cdot}$ is defined by
\begin{eqnarray}
\label{3-NS-Lie-5} \Courant{x,y,z}=\circlearrowleft_{x,y,z}\{x,y,z\}+[x,y,z],\quad \forall x,y,z\in A,
\end{eqnarray}
where $\circlearrowleft_{x,y,z}\{x,y,z\}=\{x,y,z\}+\{y,z,x\}+\{z,x,y\}$.
\end{defi}

\begin{rmk}
Let $(A,\{\cdot,\cdot,\cdot\}, [\cdot,\cdot,\cdot])$ be an NS-$3$-Lie algebra. On the one hand, if $\{\cdot,\cdot,\cdot\}=0,$
we get that $(A,[\cdot,\cdot,\cdot])$ is a $3$-Lie algebra. On the other hand, if $[\cdot,\cdot,\cdot]=0,$ then $(A,\{\cdot,\cdot,\cdot\})$
is a $3$-pre-Lie algebra which was introduced in \cite{BGS-3-Bialgebras} in the study of $3$-Lie Yang-Baxter equations.
Thus, NS-$3$-Lie algebras are generalizations of both $3$-Lie algebras and $3$-pre-Lie algebras.
\end{rmk}

\begin{thm}
Let $(A,\{\cdot,\cdot,\cdot\},[\cdot,\cdot,\cdot])$ be an NS-$3$-Lie algebra. Then $(A, \Courant{\cdot,\cdot,\cdot})$
is a $3$-Lie algebra which is called the sub-adjacent $3$-Lie algebra of $(A,\{\cdot,\cdot,\cdot\},[\cdot,\cdot,\cdot])$ and denoted by $A^c$.

Moreover, $(A;L)$ is a representation of the $3$-Lie algebra $A^c$, where the linear map $L:\otimes^2 A\rightarrow \gl(A)$ is defined by
$$ L(x,y)z=\{x,y,z\},\quad \forall~x,y,z\in A.$$
\end{thm}

\begin{proof}
By \eqref{3-NS-Lie-1}, the $3$-commutator $\Courant{\cdot,\cdot,\cdot}$ given by \eqref{3-NS-Lie-5} is skew-symmetric.

For $x_1,x_2,x_3,x_4,x_5\in A,$ by \eqref{3-NS-Lie-2}, \eqref{3-NS-Lie-3} and \eqref{3-NS-Lie-4}, we have
\begin{eqnarray*}
&&\Courant{x_1,x_2,\Courant{x_3,x_4,x_5}}-\Courant{\Courant{x_1,x_2,x_3},x_4,x_5}-\Courant{x_3,\Courant{x_1,x_2,x_4},x_5}-\Courant{x_3,x_4,\Courant{x_1,x_2,x_5}}\\
&=&\{x_1,x_2,\{x_3,x_4,x_5\}\}+\{x_1,x_2,\{x_4,x_5,x_3\}\}+\{x_1,x_2,\{x_5,x_3,x_4\}\}+\{x_1,x_2,[x_3,x_4,x_5]\}\\
&&+\{\Courant{x_3,x_4,x_5},x_1,x_2\}+\{x_2,\Courant{x_3,x_4,x_5},x_1\}+[x_1,x_2,\Courant{x_3,x_4,x_5}]\\
&&-\{x_4,x_5,\{x_1,x_2,x_3\}\}-\{x_4,x_5,\{x_2,x_3,x_1\}\}-\{x_4,x_5,\{x_3,x_1,x_2\}\}-\{x_4,x_5,[x_1,x_2,x_3]\}\\
&&-\{\Courant{x_1,x_2,x_3},x_4,x_5\}-\{x_5,\Courant{x_1,x_2,x_3},x_4\}-[\Courant{x_1,x_2,x_3},x_4,x_5]\\
&&-\{x_5,x_3,\{x_1,x_2,x_4\}\}-\{x_5,x_3,\{x_2,x_4,x_1\}\}-\{x_5,x_3,\{x_4,x_1,x_2\}\}-\{x_5,x_3,[x_1,x_2,x_4]\}\\
&&-\{x_3,\Courant{x_1,x_2,x_4},x_5\}-\{\Courant{x_1,x_2,x_4},x_3,x_5\}-[x_3,\Courant{x_1,x_2,x_4},x_5]\\
&&-\{x_3,x_4,\{x_1,x_2,x_5\}\}-\{x_3,x_4,\{x_5,x_1,x_2\}\}-\{x_3,x_4,\{x_2,x_5,x_1\}\}-\{x_3,x_4,[x_1,x_2,x_5]\}\\
&&-\{\Courant{x_1,x_2,x_5},x_3,x_4\}-\{x_4,\Courant{x_1,x_2,x_5},x_3\}-[x_3,x_4,\Courant{x_1,x_2,x_5}]\\
&=&0,
\end{eqnarray*}
which implies that $(A,\Courant{\cdot,\cdot,\cdot})$ is a $3$-Lie algebra.

By \eqref{3-NS-Lie-2} and \eqref{3-NS-Lie-3}, we have
 \begin{eqnarray*}
[L(x_1,x_2),L(x_3,x_4)](x_5)&=&L(\Courant{x_1,x_2,x_3},x_4)(x_5)+L(x_3,\Courant{x_1,x_2,x_4})(x_5),\\
 L(\Courant{x_1,x_2,x_3},x_4)(x_5)&=&\Big(L(x_1,x_2)L(x_3,x_4)+L(x_2,x_3)L(x_1,x_4)+L(x_3,x_1)L(x_2,x_4)\Big)(x_5).
 \end{eqnarray*}
 Therefore, $(A;L)$ is a representation the sub-adjacent  $3$-Lie algebra $A^c.$
\end{proof}

The following results illustrate that NS-$3$-Lie algebras can be viewed as the underlying algebraic structures of $\Phi$-twisted Rota-Baxter operators on   $3$-Lie algebras.

\begin{thm}\label{construct-3-NS-Lie algebra}
Let $T:V\rightarrow \g$ be a $\Phi$-twisted Rota-Baxter operator on a $3$-Lie algebra $(\g,[\cdot,\cdot,\cdot]_{\g})$ with respect to a representation $(V;\rho)$. Then
\begin{eqnarray}\label{twisted-Ns}
\qquad\{u_1,u_2,u_3\}=\rho(Tu_1,Tu_2)u_3 \quad and \quad [u_1,u_2,u_3]=\Phi(Tu_1,Tu_2,Tu_3), \quad\forall u_1,u_2,u_3 \in V,
\end{eqnarray}
defines an NS-$3$-Lie algebra structure on $V$.
\end{thm}

\begin{proof}
For any $u_1,u_2,u_3\in V,$ it is obvious that
\begin{eqnarray*}
&&\{u_1,u_2,u_3\}=\rho(Tu_1,Tu_2)u_3=-\rho(Tu_2,Tu_1)u_3=-\{u_2,u_1,u_3\}.
\end{eqnarray*}
Furthermore, for $u_1,u_2,u_3,u_4,u_5\in V$, by \eqref{representation-1}, \eqref{twisted Rota-Baxter operator} and \eqref{3-NS-Lie-5}, we have
\begin{eqnarray*}
&&\{u_1,u_2,\{u_3,u_4,u_5\}\}-\{u_3,u_4,\{u_1,u_2,u_5\}\}-\{\Courant{u_1,u_2,u_3},u_4,u_5\}-\{u_3,\Courant{u_1,u_2,u_4},u_5\}\\
&=&\{u_1,u_2,\rho(Tu_3,Tu_4)u_5\}-\{u_3,u_4,\rho(Tu_1,Tu_2)u_5\}\\
&&-\{\rho(Tu_1,Tu_2)u_3+\rho(Tu_2,Tu_3)u_1+\rho(Tu_3,Tu_1)u_2+\Phi(Tu_1,Tu_2,Tu_3),u_4,u_5\}\\
&&-\{u_3,\rho(Tu_1,Tu_2)u_4+\rho(Tu_2,Tu_4)u_1+\rho(Tu_4,Tu_1)u_2+\Phi(Tu_1,Tu_2,Tu_4),u_5\}\\
&=&\rho(Tu_1,Tu_2)\rho(Tu_3,Tu_4)u_5-\rho(Tu_3,Tu_4)\rho(Tu_1,Tu_2)u_5\\
&&-\rho\Big(T\rho(Tu_1,Tu_2)u_3+T\rho(Tu_2,Tu_3)u_1+T\rho(Tu_3,Tu_1)u_2+T\Phi(Tu_1,Tu_2,Tu_3),Tu_4\Big)u_5\\
&&-\rho\Big(Tu_3,T\rho(Tu_1,Tu_2)u_4+T\rho(Tu_2,Tu_4)u_1+T\rho(Tu_4,Tu_1)u_2+T\Phi(Tu_1,Tu_2,Tu_4)\Big)u_5\\
&=&0.
\end{eqnarray*}
This implies that \eqref{3-NS-Lie-2} in Definition \ref{defi-3-NS-Lie algebra} holds.
Similarly, for \eqref{3-NS-Lie-3}, by \eqref{representation-2}, we can verify that
\begin{eqnarray*}
&&\{\Courant{u_1,u_2,u_3},u_4,u_5\}-\circlearrowleft_{u_1,u_2,u_3}\{u_1,u_2,\{u_3,u_4,u_5\}\}\\
&=&\rho([Tu_1,Tu_2,Tu_3]_{\g},Tu_4)u_5-\rho(Tu_1,Tu_2)\rho(Tu_3,Tu_4)u_5\\
&&-\rho(Tu_2,Tu_3)\rho(Tu_1,Tu_4)u_5-\rho(Tu_3,Tu_1)\rho(Tu_2,Tu_4)u_5\\
&=&0.
\end{eqnarray*}
Moreover, since $\Phi$ is a $2$-cocycle of $\g$ with coefficients in $V,$ we observe that
\begin{eqnarray*}
&&[u_1,u_2,\Courant{u_3,u_4,u_5}]-\circlearrowleft_{u_3,u_4,u_5}[u_3,u_4,\Courant{u_1,u_2,u_5}]+\{u_1,u_2,[u_3,u_4,u_5]\}\\
&&-\circlearrowleft_{u_3,u_4,u_5}\{u_3,u_4,[u_1,u_2,u_5]\}\\
&=&\Phi(Tu_1,Tu_2,[Tu_3,Tu_4,Tu_5]_{\g})-\Phi([Tu_1,Tu_2,Tu_3]_{\g},Tu_4,Tu_5)\\
&&-\Phi(Tu_3,[Tu_1,Tu_2,Tu_4]_{\g},Tu_5)-\Phi(Tu_3,Tu_4,[Tu_1,Tu_2,Tu_5]_{\g})\\
&&+[Tu_1,Tu_2,\Phi(Tu_3,Tu_4,Tu_5)]_{\g}-[Tu_4,Tu_5,\Phi(Tu_1,Tu_2,Tu_3)]_{\g}\\
&&-[Tu_3,Tu_4,\Phi(Tu_1,Tu_2,Tu_5)]_{\g}-[Tu_5,Tu_3,\Phi(Tu_1,Tu_2,Tu_4)]_{\g}\\
&=&0,
\end{eqnarray*}
which implies that \eqref{3-NS-Lie-4} holds.
Hence $(V,\{\cdot,\cdot,\cdot\},[\cdot,\cdot,\cdot])$ is an NS-$3$-Lie algebra.
\end{proof}

Note that the subadjacent 3-Lie algebra of the above NS-3-Lie algebra $(V,\{\cdot,\cdot,\cdot\},[\cdot,\cdot,\cdot])$ is exactly the 3-Lie algebra given in Corollary \ref{Induce-3-Lie}.

\begin{defi}
A  homomorphism from an NS-$3$-Lie algebra $(A,\{\cdot,\cdot,\cdot\}, [\cdot,\cdot,\cdot])$ to $(A',\{\cdot,\cdot,\cdot\}', [\cdot,\cdot,\cdot]')$ is a linear map $\psi:A\lon A'$ satisfying
\begin{eqnarray}
\psi(\{x,y,z\})=\{\psi(x),\psi(y),\psi(z)\}',\\
\psi([x,y,z])=[\psi(x),\psi(y),\psi(z)]',
\end{eqnarray}
for all $x,y,z\in A.$
\end{defi}

\begin{pro}
Let $T$ and $T'$ be $\Phi$-twisted Rota-Baxter operators on a $3$-Lie algebra $\g$ with respect to a representation $(V;\rho),$ and $(V,\{\cdot,\cdot,\cdot\},[\cdot,\cdot,\cdot])$, $(V,\{\cdot,\cdot,\cdot\}',[\cdot,\cdot,\cdot]')$ be the induced NS-3-Lie algebras.  Let $(\phi,\psi)$ be a homomorphism  from $T$ to $T'$. Then $\psi$ is a homomorphism from the NS-3-Lie algebra $(V,\{\cdot,\cdot,\cdot\},[\cdot,\cdot,\cdot])$ to $(V,\{\cdot,\cdot,\cdot\}',[\cdot,\cdot,\cdot]')$.
\end{pro}
\begin{proof}
For all $u,v,w\in V,$ by \eqref{condition-1}-\eqref{condition-3} and \eqref{twisted-Ns}, we have
\begin{eqnarray*}
\psi(\{u,v,w\})&=&\psi(\rho(Tu,Tv)w)=\rho(\phi(Tu),\phi(Tv))\psi(w)=\rho(T'\psi(u),T'\psi(v))\psi(w)\\
&=&\{\psi(u),\psi(v),\psi(w)\}',\\
 \psi([u,v,w])&=&\psi(\Phi(Tu,Tv,Tw))=\Phi(\phi(Tu),\phi(Tv),\phi(Tw))=\Phi(T'\psi(u),T'\psi(v), T'\psi(w))\\
&=&[\psi(u),\psi(v),\psi(w)]',
\end{eqnarray*}
which implies that $\psi$ is a homomorphism between the induced NS-3-Lie algebras.
\end{proof}

Thus, Theorem \ref{construct-3-NS-Lie algebra} can be enhanced to a functor from the category of $\Phi$-twisted Rota-Baxter operators on a $3$-Lie algebra $\g$ with respect to a representation $(V;\rho)$ to the category of NS-3-Lie algebras.

\section{Nijenhuis operators and Reynolds operators on $3$-Lie algebras}\label{sec:N}

In this section, we study two special classes of twisted Rota-Baxter operators on 3-Lie algebras which are provided by Nijenhuis operators and Reynolds operators on $3$-Lie algebras.

\subsection{Nijenhuis operators  on   $3$-Lie algebras} In this subsection, we show that a Nijenhuis operator on a $3$-Lie algebra  gives rise to a twisted Rota-Baxter operator  on a 3-Lie algebra.

Recall from \cite{Liu-Jie-Feng} that a Nijenhuis operator on a $3$-Lie algebra $(\g,[\cdot,\cdot,\cdot]_{\g})$ is a linear map $N:\g\rightarrow\g$ satisfying
\begin{eqnarray}\label{Nijenhuis-1}
[Nx,Ny,Nz]_{\g}&=&N\Big([Nx,Ny,z]_{\g}+[x,Ny,Nz]_{\g}+[Nx,y,Nz]_{\g}\\
\nonumber&&-N[Nx,y,z]_{\g}-N[x,Ny,z]_{\g}-N[x,y,Nz]_{\g}\\
\nonumber&&+N^2[x,y,z]_{\g}\Big), \quad \forall  x,y,z\in \g.
\end{eqnarray}
A Nijenhuis operator endows the vector space $\g$ a new $3$-Lie bracket $[\cdot,\cdot,\cdot]_{N}$,
 which is defined by
\begin{eqnarray}
\label{Nijenhuis-operator}[x,y,z]_{N}&=&[Nx,Ny,z]_{\g}+[x,Ny,Nz]_{\g}+[Nx,y,Nz]_{\g}\\
\nonumber&&-N[Nx,y,z]_{\g}-N[x,Ny,z]_{\g}-N[x,y,Nz]_{\g}+N^2[x,y,z]_{\g}, \quad \forall x,y,z\in \g.
\end{eqnarray}
The 3-Lie algebra $(\g,[\cdot,\cdot,\cdot]_{N})$ will be called the {\bf deformed $3$-Lie algebra}, and denoted by $\g_{N}.$ It is obvious that $N$ is a homomorphism from the deformed $3$-Lie algebra $(\g,[\cdot,\cdot,\cdot]_{N})$ to $(\g,[\cdot,\cdot,\cdot]_{\g})$.

\begin{lem}\label{lem-representation}
Let $N$ be a Nijenhuis operator on a $3$-Lie algebra $(\g,[\cdot,\cdot,\cdot]_{\g})$. Define $\rho_{N}:\wedge^2\g_{N}\rightarrow \gl(\g)$ by
\begin{eqnarray}\label{representation-nijenhuis}
\rho_{N}(x,y)z=[Nx,Ny,z]_{\g}, \quad \forall~x,y \in \g_{N},~z\in \g.
\end{eqnarray}
Then $(\g;\rho_{N})$ is a representation of the deformed $3$-Lie algebra $\g_{N}.$
\end{lem}
\begin{proof}
For all $x_1,x_2,x_3,x_4\in \g_{N},x\in \g,$ by \eqref{eq:jacobi1}, \eqref{Nijenhuis-1} and \eqref{Nijenhuis-operator}, we have
\begin{eqnarray*}
&&\Big(\rho_{N}(x_1,x_2)\rho_{N}(x_3,x_4)-\rho_{N}(x_3,x_4)\rho_{N}(x_1,x_2)-\rho_{N}([x_1,x_2,x_3]_{N},x_4)-\rho_{N}(x_3,[x_1,x_2,x_4]_{N})\Big)(x)\\
&=&\rho_{N}(x_1,x_2)[Nx_3,Nx_4,x]_{\g}-\rho_{N}(x_3,x_4)[Nx_1,Nx_2,x]_{\g}\\
&&-[N[x_1,x_2,x_3]_{N},Nx_4,x]_{\g}-[Nx_3,N[x_1,x_2,x_4]_{N},x]_{\g}\\
&=&[Nx_1,Nx_2,[Nx_3,Nx_4,x]_{\g}]_{\g}-[Nx_3,Nx_4,[Nx_1,Nx_2,x]_{\g}]_{\g}\\
&&-[[Nx_1,Nx_2,Nx_3]_{\g},Nx_4,x]_{\g}-[Nx_3,[Nx_1,Nx_2,Nx_4]_{\g},x]_{\g}\\
&=&0,
\end{eqnarray*}
which implies that \eqref{representation-1} in Definition \ref{defi-representation}  holds.

Similarly, we can show that \eqref{representation-2} also holds.
Thus, $(\g;\rho_{N})$ is a representation of the deformed $3$-Lie algebra $(\g,[\cdot,\cdot,\cdot]_{N})$.
\end{proof}
\emptycomment{
&&(\rho_{N}(x_1,[x_2,x_3,x_4]_{N})-\rho_{N}(x_3,x_4)\rho_{N}(x_1,x_2)+\rho_{N}(x_2,x_4)\rho_{N}(x_1,x_3)-\rho_{N}(x_2,x_3)\rho_{N}(x_1,x_4))(x)\\
&=&[Nx_1,[Nx_2,Nx_3,Nx_4]_{\g},x]_{\g}-[Nx_3,Nx_4,[Nx_1,Nx_2,x]_{\g}]_{\g}\\
&&+[Nx_2,Nx_4,[Nx_1,Nx_3,x]_{\g}]_{\g}-[Nx_2,Nx_3,[Nx_1,Nx_4,x]_{\g}]_{\g}\\
&=&0.}
\begin{thm}\label{pro-Nijenhuis}
Let $N$ be a Nijenhuis operator on a $3$-Lie algebra $(\g,[\cdot,\cdot,\cdot]_{\g})$. Define the map $\Phi:\wedge^3\g_{N}\rightarrow \g$ by
 \begin{eqnarray}
\label{cocycle-nijenhuis}\Phi(x,y,z)=-N([Nx,y,z]_{\g}+[x,Ny,z]_{\g}+[x,y,Nz]_{\g})+N^2[x,y,z]_{\g},\quad \forall x,y,z\in \g_{N}.
 \end{eqnarray}
 Then $\Phi$ is a $2$-cocycle of $\g_{N}$ with coefficients in $(\g;\rho_{N}),$ and the identity map $\Id:\g\rightarrow\g_{N}$ is a $\Phi$-twisted Rota-Baxter operator on $\g_{N}$ with respect to the representation $(\g;\rho_{N}).$
\end{thm}
\begin{proof}
For all $x_1,x_2,x_3,x_4,x_5\in \g_{N},$ by a direct calculation, we have
{\footnotesize
\begin{eqnarray*}
&&(\partial\Phi)(x_1,x_2,x_3,x_4,x_5)\\
&=&\rho_{N}(x_1,x_2)\Phi(x_3,x_4,x_5)-\rho_{N}(x_3,x_4)\Phi(x_1,x_2,x_5)-\rho_{N}(x_4,x_5)\Phi(x_1,x_2,x_3)-\rho_{N}(x_5,x_3)\Phi(x_1,x_2,x_4)\\
&&\Phi([x_3,x_4,x_5]_{N},x_1,x_2)-\Phi([x_1,x_2,x_3]_{N},x_4,x_5)-\Phi([x_1,x_2,x_4]_{N},x_5,x_3)-\Phi([x_1,x_2,x_5]_{N},x_3,x_4)\\
&=&[Nx_1,Nx_2,N^2[x_3,x_4,x_5]_{\g}-N[Nx_3,x_4,x_5]_{\g}-N[x_3,Nx_4,x_5]_{\g}-N[x_3,x_4,Nx_5]_{\g}]_{\g}\\
&&+[Nx_3,Nx_4,N[Nx_1,x_2,x_5]_{\g}+N[x_1,Nx_2,x_5]_{\g}+N[x_1,x_2,Nx_5]_{\g}-N^2[x_1,x_2,x_5]_{\g}]_{\g}\\
&&+[Nx_4,Nx_5,N[Nx_1,x_2,x_3]_{\g}+N[x_1,Nx_2,x_3]_{\g}+N[x_1,x_2,Nx_3]_{\g}-N^2[x_1,x_2,x_3]_{\g}]_{\g}\\
&&+[Nx_5,Nx_3,N[Nx_1,x_2,x_4]_{\g}+N[x_1,Nx_2,x_4]_{\g}+N[x_1,x_2,Nx_4]_{\g}-N^2[x_1,x_2,x_4]_{\g}]_{\g}\\
&&+N\Big(N[[x_3,x_4,x_5]_{N},x_1,x_2]_{\g}-[N[x_3,x_4,x_5]_{N},x_1,x_2]_{\g}-[[x_3,x_4,x_5]_{N},Nx_1,x_2]_{\g}-[[x_3,x_4,x_5]_{N},x_1,Nx_2]_{\g}\Big)\\
&&+N\Big([N[x_1,x_2,x_3]_{N},x_4,x_5]_{\g}+[[x_1,x_2,x_3]_{N},Nx_4,x_5]_{\g}+[[x_1,x_2,x_3]_{N},x_4,Nx_5]_{\g}-N[[x_1,x_2,x_3]_{N},x_4,x_5]_{\g}\Big)\\
&&+N\Big([N[x_1,x_2,x_4]_{N},x_5,x_3]_{\g}+[[x_1,x_2,x_4]_{N},Nx_5,x_3]_{\g}+[[x_1,x_2,x_4]_{N},x_5,Nx_3]_{\g}-N[[x_1,x_2,x_4]_{N},x_5,x_3]_{\g}\Big)\\
&&+N\Big([N[x_1,x_2,x_5]_{N},x_3,x_4]_{\g}+[[x_1,x_2,x_5]_{N},Nx_3,x_4]_{\g}+[[x_1,x_2,x_5]_{N},x_3,Nx_4]_{\g}-N[[x_1,x_2,x_5]_{N},x_3,x_4]_{\g}\Big)\\
&=&0.
\end{eqnarray*}}
Thus, we deduce that $\Phi$ is a $2$-cocycle of $\g_{N}$ with coefficients in $(\g;\rho_{N}).$

Moreover, by \eqref{representation-nijenhuis} and \eqref{cocycle-nijenhuis}, it is easy to prove that \eqref{twisted Rota-Baxter operator} is equivalent to \eqref{Nijenhuis-operator}, which implies that the identity map $\Id:\g\rightarrow\g_{N}$ is a $\Phi$-twisted Rota-Baxter operator.
The proof is finished.
\end{proof}

\begin{cor}\label{3-NS-Nijenhuis}
Let $(\g,[\cdot,\cdot,\cdot]_{\g})$ be a $3$-Lie algebra and $N:\g\rightarrow \g$ be a Nijenhuis operator. Then there exists an NS-$3$-Lie algebra on $\g$ given by
\begin{eqnarray}
\label{N-NS-1}\{x,y,z\}&=&[Nx,Ny,z]_{\g},\\
\label{N-NS-2}~[x,y,z]&=&-N([Nx,y,z]_{\g}+[x,Ny,z]_{\g}+[x,y,Nz]_{\g})+N^2[x,y,z]_{\g}, \quad \forall x,y,z\in \g.
\end{eqnarray}
\end{cor}
\begin{proof}
By Theorem \ref{pro-Nijenhuis},  the identity map $\Id:\g\rightarrow\g_{N}$ is a $\Phi$-twisted Rota-Baxter operator on $\g_{N}$ with respect to the representation $(\g;\rho_N)$. By Theorem \ref{construct-3-NS-Lie algebra}, a $\Phi$-twisted Rota-Baxter operator on a $3$-Lie algebra can induce an NS-$3$-Lie algebra.
Hence the conclusion follows.
\end{proof}
\begin{ex}{\rm
Consider the $3$-dimensional $3$-Lie algebra $(\g,[\cdot,\cdot,\cdot]_{\g})$ given with respect to a basis $\{e_1,e_2,e_3\}$ by
$$[e_1,e_2,e_3]_{\g}=e_1.$$
Thanks to Theorem 3.10 in {\rm\cite{Liu-Jie-Feng}}, any linear transformation $N$ on $\g$ is a Nijenhuis operator. Suppose
\begin{eqnarray*}
N=(a_{ij})_{3\times3}=\left(\begin{array}{ccc}
 a_{11}&a_{12}&a_{13}\\
 a_{21}&a_{22}&a_{23}\\
 a_{31}&a_{32}&a_{33}
 \end{array}\right).
 \end{eqnarray*}
By Corollary \ref{3-NS-Nijenhuis}, $(\g,\{\cdot,\cdot,\cdot\}, [\cdot,\cdot,\cdot])$ is a $3$-dimensional NS-$3$-Lie algebra, the brackets $\{\cdot,\cdot,\cdot\}:\otimes^3 \g\rightarrow \g$ and $[\cdot,\cdot,\cdot]:\wedge^3 \g\rightarrow \g$ are given by
 \begin{eqnarray*}
~\{e_i,e_j,e_k\}&=&(-1)^{\tau(ijk)}M_{kk}e_1,\quad i\neq j\neq k,\\
~\{e_1,e_2,e_1\}&=&-\{e_2,e_1,e_1\}=M_{13}e_1,\qquad\{e_1,e_3,e_1\}=-\{e_3,e_1,e_1\}=M_{12}e_1,\\
~\{e_2,e_1,e_2\}&=&-\{e_1,e_2,e_2\}=M_{23}e_1,\qquad\{e_2,e_3,e_2\}=-\{e_3,e_2,e_2\}=-M_{21}e_1,\\
~\{e_3,e_1,e_3\}&=&-\{e_1,e_3,e_3\}=-M_{32}e_1,~~\quad\{e_3,e_2,e_3\}=-\{e_2,e_3,e_3\}=-M_{31}e_1,\\
~[e_i,e_j,e_k]&=&(-1)^{\tau(ijk)}(-M_{22}-M_{33}e_1-M_{12}e_2+M_{13}e_3),\quad i\neq j\neq k,
 \end{eqnarray*}
and all the other brackets are zero, where the terms $\tau(ijk)$ is the inverse table permutation of $i,j,k$ and $M_{ij}$ is the complementary minor of element $a_{ij}.$

In fact, by \eqref{N-NS-1}, we have
 \begin{eqnarray*}
 \{e_i,e_j,e_k\}&=&[Ne_i,Ne_i,e_k]_{\g}\\
 &=&[a_{1i}e_1+a_{2i}e_2+a_{3i}e_3,a_{1j}e_1+a_{2j}e_2+a_{3j}e_3,e_k]_{\g}\\
 &=&[a_{1i}e_1,a_{2j}e_2,e_k]_{\g}+[a_{1i}e_1,a_{3j}e_3,e_k]_{\g}+[a_{2i}e_2,a_{1j}e_1,e_k]_{\g}\\
 &&+[a_{2i}e_2,a_{3j}e_3,e_k]_{\g}+[a_{3i}e_3,a_{1j}e_1,e_k]_{\g}+[a_{3i}e_3,a_{2j}e_2,e_k]_{\g}.
 \end{eqnarray*}
When $k=1,$ we have
\begin{itemize}
 \item if $i=j,$ we get $\{e_i,e_j,e_1\}=0;$
\item if $i\neq j\neq k,$ we get
$\{e_i,e_j,e_1\}=\left|\begin{array}{cc}
 a_{2i}&a_{2j}\\
 a_{3i}&a_{3j}
 \end{array}\right|e_1=(-1)^{\tau(ij1)}M_{11}e_1;$
\item if $i=k=1, j=2~\mbox{or}~3,$ we get
 $$\{e_1,e_2,e_1\}=\left|\begin{array}{cc}
 a_{21}&a_{22}\\
 a_{31}&a_{32}
 \end{array}\right|e_1=M_{13}e_1,\qquad \{e_1,e_3,e_1\}=\left|\begin{array}{cc}
 a_{21}&a_{23}\\
 a_{31}&a_{33}
 \end{array}\right|e_1=M_{12}e_1.$$
\end{itemize}
Similarly, we can obtain all the formulas listed above.

}
\end{ex}

\subsection{Reynolds operators on $3$-Lie algebras}

In this subsection, we introduce the notion of a Reynolds operator  on a $3$-Lie algebra, which turns out to be a special twisted Rota-Baxter operator. The relation between Reynolds operators and derivations  on  $3$-Lie algebras are investigated.

\begin{defi}
Let $(\g,[\cdot,\cdot,\cdot]_{\g})$ be a $3$-Lie algebra. A linear map $R:\g\rightarrow\g$ is called a {\bf Reynolds operator} if
\begin{equation}
\label{Reynolds-operator}[Rx,Ry,Rz]_{\g}=R\Big([Rx,Ry,z]_{\g}+[x,Ry,Rz]_{\g}+[Rx,y,Rz]_{\g}-[Rx,Ry,Rz]_{\g}\Big), \quad \forall~x,y,z\in \g.
\end{equation}
Moreover, a $3$-Lie algebra $\g$ with a Reynolds operator $R$ is called a {\bf Reynolds $3$-Lie algebra}. We denote it by $(\g,[\cdot,\cdot,\cdot]_{\g},R).$
\end{defi}

\begin{defi}
Let $(\g_1,[\cdot,\cdot,\cdot]_{\g_1},R_1)$ and $(\g_2,[\cdot,\cdot,\cdot]_{\g_2},R_2)$ be Reynolds $3$-Lie algebras. A linear map $\phi:\g_1\rightarrow \g_2$ is called a {\bf homomorphism} of Reynolds $3$-Lie algebras if $\phi$ is a $3$-Lie algebra homomorphism and $\phi\circ R_1=R_2\circ \phi$.
\end{defi}

\begin{thm}\label{Reynolds-3-Lie algebra}
Let $(\g,[\cdot,\cdot,\cdot]_{\g},R)$ be a Reynolds $3$-Lie algebra. Define a multiplication $[\cdot,\cdot,\cdot]_{R}$ on $\g$ by
\begin{equation}
\label{induce-3-Lie}[x,y,z]_{R}=[Rx,Ry,z]_{\g}+[x,Ry,Rz]_{\g}+[Rx,y,Rz]_{\g}-[Rx,Ry,Rz]_{\g},\quad \forall x,y,z\in \g.
\end{equation}
Then
\begin{itemize}
\item[{\rm (a)}] $[Rx,Ry,Rz]_{\g}=R([x,y,z]_{R})$;
\item[{\rm (b)}] $(\g,[\cdot,\cdot,\cdot]_{R})$ is a $3$-Lie algebra;
\item[{\rm (c)}] $(\g,[\cdot,\cdot,\cdot]_{R},R)$ is a Reynolds $3$-Lie algebra;
\item[{\rm (d)}] $R$ is a Reynolds $3$-Lie algebra homomorphism from $(\g,[\cdot,\cdot,\cdot]_{R},R)$ to $(\g,[\cdot,\cdot,\cdot]_{\g},R)$.
\end{itemize}
\end{thm}
\begin{proof}
\item[{\rm (a)}.] It follows directly from  \eqref{Reynolds-operator}.
\item[{\rm (b)}.] Obviously, $[\cdot,\cdot,\cdot]_{R}$ is skew-symmetric.
For $x_1,x_2,x_3,x_4,x_5\in \g,$ by \eqref{eq:jacobi1} and \eqref{induce-3-Lie}, we have
\begin{eqnarray*}
&&[x_1,x_2,[x_3,x_4,x_5]_{R}]_{R}-[[x_1,x_2,x_3]_{R},x_4,x_5]_{R}-[x_3,[x_1,x_2,x_4]_{R},x_5]_{R}-[x_3,x_4,[x_1,x_2,x_5]_{R}]_{R}\\
&=&[Rx_1,Rx_2,[Rx_3,Rx_4,x_5]_{\g}]_{\g}+[Rx_1,Rx_2,[x_3,Rx_4,Rx_5]_{\g}]_{\g}+[Rx_1,Rx_2,[Rx_3,x_4,x_5]_{\g}]_{\g}\\
&&+[x_1,Rx_2,[Rx_3,Rx_4,Rx_5]_{\g}]_{\g}+[Rx_1,x_2,[Rx_3,Rx_4,Rx_5]_{\g}]_{\g}\\
&&-[[Rx_1,Rx_2,x_3]_{\g},Rx_4,Rx_5]_{\g}-[[Rx_1,x_2,Rx_3]_{\g},Rx_4,Rx_5]_{\g}-[[x_1,Rx_2,Rx_3]_{\g},Rx_4,Rx_5]_{\g}\\
&&-[[Rx_1,Rx_2,Rx_3]_{\g},Rx_4,x_5]_{\g}-[[Rx_1,Rx_2,Rx_3]_{\g},x_4,Rx_5]_{\g}\\
&&-[Rx_3,[Rx_1,Rx_2,x_4]_{\g},Rx_5]_{\g}-[Rx_3,[x_1,Rx_2,Rx_4]_{\g},Rx_5]_{\g}-[Rx_3,[Rx_1,x_2,Rx_4]_{\g},Rx_5]_{\g}\\
&&-[Rx_3,[Rx_1,Rx_2,Rx_4]_{\g},x_5]_{\g}-[x_3,[Rx_1,Rx_2,Rx_4]_{\g},Rx_5]_{\g}\\
&&-[Rx_3,Rx_4,[Rx_1,Rx_2,x_5]_{\g}]_{\g}-[Rx_3,Rx_4,[Rx_1,x_2,Rx_5]_{\g}]_{\g}-[Rx_3,Rx_4,[x_1,Rx_2,Rx_5]_{\g}]_{\g}\\
&&-[Rx_3,x_4,[Rx_1,Rx_2,Rx_5]_{\g}]_{\g}-[x_3,Rx_4,[Rx_1,Rx_2,Rx_5]_{\g}]_{\g}\\
&=&0.
\end{eqnarray*}
Thus, $(\g,[\cdot,\cdot,\cdot]_{R})$ is a $3$-Lie algebra.
\item[{\rm (c)}.] For $x,y,z\in \g,$ by \eqref{induce-3-Lie} and Item {\rm (a)}, we have
\begin{eqnarray*}
[Rx,Ry,Rz]_{R}&=&[R^2x,R^2y,Rz]_{\g}+[Rx,R^2y,R^2z]_{\g}+[R^2x,Ry,R^2z]_{\g}-[R^2x,R^2y,R^2z]_{\g}\\
&=&R([Rx,Ry,z]_{R}+[x,Ry,Rz]_{R}+[Rx,y,Rz]_{R}-[Rx,Ry,Rz]_{R}),
\end{eqnarray*}
which implies that $R$ is a Reynolds operator on the $3$-Lie algebra $(\g,[\cdot,\cdot,\cdot]_{R}).$
\item[{\rm (d)}.] By Item {\rm (a)}, $R$ is a $3$-Lie algebra homomorphism. Moreover, $R$ commutes with itself. Therefore,
$R$ is a Reynolds $3$-Lie algebra homomorphism from $(\g,[\cdot,\cdot,\cdot]_{R},R)$ to $(\g,[\cdot,\cdot,\cdot]_{\g},R).$
\end{proof}

The following results give the relation between Reynolds operators and twisted Rota-Baxter operators on 3-Lie algebras.

\begin{pro}\label{twisted-Reynolds}
Let $R$ be a Reynolds operator on a $3$-Lie algebra $(\g,[\cdot,\cdot,\cdot]_{\g})$. Then $R$ is a $\Phi$-twisted Rota-Baxter operator on $\g$ with respect to the adjoint representation $(\g;\ad)$, where $\Phi\in \Hom(\wedge^3\g,\g)$ is defined by
$$\Phi(x,y,z)=-[x,y,z]_{\g},\quad \forall x,y,z\in \g.$$
\end{pro}
\begin{proof}
Let $(\g,[\cdot,\cdot,\cdot]_{\g})$ be a $3$-Lie algebra. By \eqref{2-cocycle}, the $3$-Lie bracket $[\cdot,\cdot,\cdot]_{\g}$ is a $2$-cocycle with coefficients in the adjoint representation $(\g;\ad)$, which implies that $R$ is a $\Phi$-twisted Rota-Baxter operator on $\g$ with respect to the adjoint representation.
\end{proof}

As Reynolds operators on $3$-Lie algebras are particular $\Phi$-twisted Rota-Baxter operators, the following conclusion is obvious.

\begin{cor}\label{Reynolds-3-NS-algebra}
Let $R:\g\rightarrow\g$ be a Reynolds operator on a $3$-Lie algebra $(\g,[\cdot,\cdot,\cdot]_{\g})$. Then
\begin{eqnarray*}
\{x,y,z\}=[Rx,Ry,z]_{\g}\quad\mbox{and} \quad[x,y,z]=-[Rx,Ry,Rz]_{\g},\quad \forall x,y,z\in \g,
\end{eqnarray*}
 defines an NS-$3$-Lie algebra structure on $\g.$
\end{cor}
Next, we illustrate the relationship between Reynolds operators and derivations on a $3$-Lie algebra.
\begin{pro}
Let $R:\g\rightarrow \g$ be a Reynolds operator on a $3$-Lie algebra $(\g,[\cdot,\cdot,\cdot]_{\g})$. If $R$ is invertible, then $(R^{-1}-\frac{1}{2}\Id):\g\rightarrow\g$ is a derivation on the $3$-Lie algebra $\g$, where $\Id$ is the identity operator.
\end{pro}
\begin{proof}
Let $R:\g\rightarrow \g$ be an invertible Reynolds operator on $\g$.
By \eqref{Reynolds-operator}, we have
\begin{eqnarray*}
R^{-1}[x,y,z]_{\g}=[R^{-1}x,y,z]_{\g}+[x,R^{-1}y,z]_{\g}+[x,y,R^{-1}z]_{\g}-[x,y,z]_{\g},\quad \forall x,y,z\in \g,
\end{eqnarray*}
which implies that, for all $ x,y,z\in \g,$
\begin{eqnarray*}
(R^{-1}-\frac{1}{2}\Id)[x,y,z]_{\g}=[(R^{-1}-\frac{1}{2}\Id)x,y,z]_{\g}+[x,(R^{-1}-\frac{1}{2}\Id)y,z]_{\g}+[x,y,(R^{-1}-\frac{1}{2}\Id)z]_{\g}.
\end{eqnarray*}
This shows that $(R^{-1}-\frac{1}{2}\Id):\g\rightarrow\g$ is a derivation on the $3$-Lie algebra $\g$.
\end{proof}
Conversely, we can derive a Reynolds operator on a $3$-Lie algebra from a derivation.
\begin{pro}
Let $D:\g\rightarrow \g$ be a derivation on a $3$-Lie algebra $(\g,[\cdot,\cdot,\cdot]_{\g})$. If $(D+\frac{1}{2}\Id):\g\rightarrow\g$ is invertible, then $(D+\frac{1}{2}\Id)^{-1}$ is a Reynolds operator.
\end{pro}
\begin{proof}
Let $D:\g\rightarrow \g$ be a derivation on a $3$-Lie algebra $(\g,[\cdot,\cdot,\cdot]_{\g}).$
By \eqref{eq:der}, for all $u,v,w \in \g$, we have
\begin{eqnarray*}
(D+\frac{1}{2}\Id)[u,v,w]_{\g}=[(D+\frac{1}{2}\Id)u,v,w]_{\g}+[u,(D+\frac{1}{2}\Id)v,w]_{\g}+[u,v,(D+\frac{1}{2}\Id)w]_{\g}-[u,v,w]_{\g}.
\end{eqnarray*}
For convenience, we denote $P=D+\frac{1}{2}\Id.$
If $P$ is invertible, we put $Pu=x,Pv=y~\mbox{and}~Pw=z,$
we get
\begin{eqnarray*}
[P^{-1}x,P^{-1}y,P^{-1}z]_{\g}&=&P^{-1}\Big([x,P^{-1}y,P^{-1}z]_{\g}+[P^{-1}x,y,P^{-1}z]_{\g}\\
&&+[P^{-1}x,P^{-1}y,z]_{\g}-[P^{-1}x,P^{-1}y,P^{-1}z]_{\g}\Big),
\end{eqnarray*}
which implies that $P^{-1}$ is a Reynolds operator. The proof is finished.
\end{proof}

At the end of this section, we provide some examples of Reynolds operators on infinite dimensional $3$-Lie algebras.

In {\rm\cite{Baiwu}}, the authors construct a class of infinite dimensional $3$-Lie algebras by Laurent polynomials.
\begin{ex}{\rm
Let $A=F[t^{-1},t]$ be the Laurent polynomials over the field of complex numbers.
 Then $A$ is an infinite dimensional $3$-Lie algebra with the multiplication,
\begin{equation}
[t^{l},t^{m},t^{n}]_{A}=
\left|\begin{array}{ccc}
(-1)^l&(-1)^m&(-1)^n\\
1&1&1\\
l&m&n\\
\end{array}\right|t^{l+m+n-1},
\end{equation}
for all $l,m,n\in \mathbb{Z},~t^{l},t^{m},t^{n}\in A$.

Let $R:A\rightarrow A$ be a linear map defined by $R(t^{l})=\frac{1}{l}t^{l},~l\in \mathbb{Z}.$
For any $l,m,n\in \mathbb{Z},$ we obtain
\begin{eqnarray}\label{Reynold-1}
[R(t^{l}),R(t^{m}),R(t^{n})]_{A}=\frac{1}{lmn}[t^{l},t^{m},t^{n}]_{A}.
\end{eqnarray}
On the other hand,
\begin{eqnarray}\label{Reynold-2}
&&R([R(t^{l}),R(t^{m}),t^{n}]_{A}+[t^{l},R(t^{m}),R(t^{n})]_{A}+[R(t^{l}),t^{m},R(t^{n})]_{A}-[R(t^{l}),R(t^{m}),R(t^{n})]_{A})\\
\nonumber&=&R((\frac{1}{lm}+\frac{1}{mn}+\frac{1}{ln}-\frac{1}{lmn})[t^{l},t^{m},t^{n}]_{A})\\
\nonumber&=&(\frac{1}{lm}+\frac{1}{mn}+\frac{1}{ln}-\frac{1}{lmn})\frac{1}{l+m+n-1}[t^{l},t^{m},t^{n}]_{A}\\
\nonumber&=&\frac{1}{lmn}[t^{l},t^{m},t^{n}]_{A}.
\end{eqnarray}
It follows from \eqref{Reynold-1} and \eqref{Reynold-2} that $R$ is a Reynolds operator on $A.$

By Item {\rm (b)} and Item {\rm (c)} in  Theorem \ref{Reynolds-3-Lie algebra}, $(A,[\cdot,\cdot,\cdot]_{R})$ is a $3$-Lie algebra and $(A,[\cdot,\cdot,\cdot]_{R},R)$ is a Reynolds $3$-Lie algebra, where the $3$-Lie bracket $[\cdot,\cdot,\cdot]_{R}$ is defined by
\begin{eqnarray*}
 &&[t^{l},t^{m},t^{n}]_{R}:\\
&=&[R(t^{l}),R(t^{m}),t^{n}]_{A}+[t^{l},R(t^{m}),R(t^{n})]_{A}+[R(t^{l}),t^{m},R(t^{n})]_{A}-[R(t^{l}),R(t^{m}),R(t^{n})]_{A}\\
&=&\frac{l+m+n-1}{lmn}
\left|\begin{array}{ccc}
(-1)^l&(-1)^m&(-1)^n\\
1&1&1\\
l&m&n\\
\end{array}\right|t^{l+m+n-1}.
\end{eqnarray*}

According to Corollary \ref{Reynolds-3-NS-algebra}, the Reynolds operator induces an  NS-$3$-Lie algebra $(A,\{\cdot,\cdot,\cdot\},[\cdot,\cdot,\cdot])$, where
\begin{eqnarray*}
\{t^{l},t^{m},t^{n}\}&=&-\{t^{l},t^{m},t^{n}\}=\frac{1}{lm}[t^{l},t^{m},t^{n}]_{A},\\
~[t^{l},t^{m},t^{n}]&=&\frac{1}{lmn}[t^{l},t^{m},t^{n}]_{A}.
\end{eqnarray*}
}
\end{ex}

In {\rm\cite{S.Chakrabortty}}, the authors constructed an infinite-dimensional $3$-Lie algebra, called the $\omega_{\infty}$~$3$-Lie algebra by applying a double scaling limits on the generators of the $W_{\infty}$~algebra {\rm\cite{Pope}}.
\begin{ex}{\rm
Let $\{\omega^{a}_{m}|a,m\in \mathbb{Z}\}$ generate an infinite-dimensional vector space. Define
\begin{equation}
[\omega^{a}_m,\omega^{b}_n,\omega^{c}_p]=
\left|\begin{array}{ccc}
1&1&1\\
m&n&p\\
a&b&c\\
\end{array}\right|\omega^{a+b+c+1}_{m+n+p},\quad\forall~m,n,p,a,b,c\in \mathbb{Z}.
\end{equation}
Then $\{\omega^{a}_{m}|a,m\in \mathbb{Z}\}$ generate an infinite-dimensional $3$-Lie algebra which is called the $\omega_{\infty}$~$3$-Lie algebra. We denote the $3$-Lie algebra by
$\langle\omega^{a}_{m}\rangle.$

Let $R:\langle\omega^{a}_{m}\rangle\rightarrow \langle\omega^{a}_{m}\rangle$ be a linear map defined by
$$R(\omega^{a}_m)=\frac{1}{m+a+1}\omega^{a}_m,\quad m,a\in \mathbb{Z}.$$
We observe that, for any $m,n,p,a,b,c\in \mathbb{Z},$
\begin{eqnarray}\label{Reynold-3}
[R(\omega^{a}_m),R(\omega^{b}_n),R(\omega^{c}_p)]=\frac{1}{(m+a+1)(n+b+1)(p+c+1)}[\omega^{a}_m,\omega^{b}_n,\omega^{c}_p].
\end{eqnarray}
On the other hand,
\begin{eqnarray}\label{Reynold-4}
&&R([R(\omega^{a}_m),R(\omega^{b}_n),\omega^{c}_p]+[\omega^{a}_m,R(\omega^{b}_n),R(\omega^{c}_p)]+[R(\omega^{a}_m),\omega^{b}_n,R(\omega^{c}_p)]-[R(\omega^{a}_m),R(\omega^{b}_n),R(\omega^{c}_p)])\\
\nonumber&=&R((\frac{1}{(m+a+1)(n+b+1)}+\frac{1}{(n+b+1)(p+c+1)}+\frac{1}{(m+a+1)(p+c+1)}\\
\nonumber&&-\frac{1}{(m+a+1)(n+b+1)(p+c+1)})[\omega^{a}_m,\omega^{b}_n,\omega^{c}_p])\\
\nonumber&=&(\frac{1}{(m+a+1)(n+b+1)}+\frac{1}{(n+b+1)(p+c+1)}+\frac{1}{(m+a+1)(p+c+1)}\\
\nonumber&&-\frac{1}{(m+a+1)(n+b+1)(p+c+1)})\frac{1}{m+n+p+a+b+c+2}[\omega^{a}_m,\omega^{b}_n,\omega^{c}_p]\\
\nonumber&=&\frac{1}{(m+a+1)(n+b+1)(p+c+1)}[\omega^{a}_m,\omega^{b}_n,\omega^{c}_p].
\end{eqnarray}
By \eqref{Reynold-3} and \eqref{Reynold-4}, $R$ is a Reynolds operator on $\langle\omega^{a}_{m}\rangle.$

Similarly, $(\langle\omega^{a}_{m}\rangle,[\cdot,\cdot,\cdot]_{R})$ is a $3$-Lie algebra and $(\langle\omega^{a}_{m}\rangle,[\cdot,\cdot,\cdot]_{R},R)$ is a Reynolds $3$-Lie algebra, where the $3$-Lie bracket $[\cdot,\cdot,\cdot]_{R}$ is given by
\begin{eqnarray*}
[\omega^{a}_m,\omega^{b}_n,\omega^{c}_p]_{R}:&=&[R(\omega^{a}_m),R(\omega^{b}_n),\omega^{c}_p]+[\omega^{a}_m,R(\omega^{b}_n),R(\omega^{c}_p)]\\
&&+[R(\omega^{a}_m),\omega^{b}_n,R(\omega^{c}_p)]-[R(\omega^{a}_m),R(\omega^{b}_n),R(\omega^{c}_p)]\\
&=&\frac{m+n+p+a+b+c+2}{(m+a+1)(n+b+1)(p+c+1)}
\left|\begin{array}{ccc}
1&1&1\\
m&n&p\\
a&b&c\\
\end{array}\right|\omega^{a+b+c+1}_{m+n+p}.
\end{eqnarray*}
\emptycomment{
 And the Reynolds operator induces a  NS-$3$-Lie algebra $(\omega,\{\cdot,\cdot,\cdot\},[\cdot,\cdot,\cdot])$, where
\begin{eqnarray*}
\{\omega^{a}_m,\omega^{b}_n,\omega^{c}_p\}=[R\omega^{a}_m,R\omega^{b}_n,\omega^{c}_p]=
\frac{1}{(m+a+1)(n+b+1)}
\left|\begin{array}{ccc}
1&1&1\\
m&n&p\\
a&b&c\\
\end{array}\right|\omega^{a+b+c+1}_{m+n+p},
\end{eqnarray*}
and
\begin{eqnarray*}
[\omega^{a}_m,\omega^{b}_n,\omega^{c}_p]=-[R\omega^{a}_m,R\omega^{b}_n,R\omega^{c}_p]=
-\frac{1}{(m+a+1)(n+b+1)(p+c+1)}
\left|\begin{array}{ccc}
1&1&1\\
m&n&p\\
a&b&c\\
\end{array}\right|\omega^{a+b+c+1}_{m+n+p}.
\end{eqnarray*}}
}
\end{ex}

 \end{document}